\documentclass{amsart}

\usepackage{amsmath,amsthm,amssymb,mathtools,amscd,color}
\usepackage{shuffle}
\usepackage[all]{xy}
\usepackage{setspace}
\allowdisplaybreaks
\newtheorem{theorem}{Theorem}[section]
\newtheorem{proposition}[theorem]{Proposition}

\newtheorem{corollary}[theorem]{Corollary}

\newtheorem{question}[theorem]{Question}

\theoremstyle{definition}
\newtheorem{definition}[theorem]{Definition}
\newtheorem{example}[theorem]{Example}

\theoremstyle{remark}
\newtheorem{remark}[theorem]{Remark}
\numberwithin{equation}{section}

\newcommand{\Z}{{\mathbb Z}}

\newcommand{\C}{{\mathbb C}}
\newcommand{\Q}{{\mathbb Q}}
\newcommand{\F}{{\mathbb F}}

\newcommand{\bbF}{\mathbb{F}}
\newcommand{\cA}{\mathcal{A}}
\newcommand{\cP}{\mathcal{P}}
\newcommand{\frp}{\mathfrak{p}}
\newcommand{\bP}{\boldsymbol{P}}
\newcommand{\ba}{\boldsymbol{a}}

\newcommand{\e}{\mathbf{e}}
\newcommand{\f}{\mathbf{f}}

\newcommand{\ev}{\mathrm{ev}}
\newcommand{\id}{\mathrm{id}}

\DeclareMathOperator{\Gal}{Gal}
\DeclareMathOperator{\Fun}{Fun}
\DeclareMathOperator{\Rec}{Rec}
\DeclareMathOperator{\disc}{disc}
\DeclareMathOperator{\Tr}{Tr}
\DeclareMathOperator{\Spl}{Spl}

\title[Finite algebraic numbers and the law of decomposition of primes]%
{The ring of finite algebraic numbers and its application to the law of decomposition of primes}
\author[J.~Rosen]{Julian Rosen}
\address[J.~Rosen]{}
\email{julianrosen@gmail.com}

\author[Y.~Takeyama]{Yoshihiro Takeyama}
\address[Y.~Takeyama]{Department of Mathematics\\
	Institute of Pure and Applied Sciences\\
	University of Tsukuba\\
	Tsukuba, Ibaraki 305-8571\\
	Japan}
\email{takeyama@math.tsukuba.ac.jp}

\author[K.~Tasaka]{Koji Tasaka}
\address[K.~Tasaka]{Department of Information Science and Technology\\
	Aichi Prefectural University\\
	Nagakute-city, Aichi, 480-1198\\
	Japan}
\email{tasaka@ist.aichi-pu.ac.jp}

\author[S.~Yamamoto]{Shuji Yamamoto}
\address[S.~Yamamoto]{Department of Mathematics, Faculty of Science and Technology\\
    Keio University\\
    3-14-1 Hiyoshi, Kouhoku-ku, Yokohama 223-8522\\
    Japan}
\email{yamashu@math.keio.ac.jp}

\keywords{Algebraic numbers, Linear recurrent sequence modulo primes, Mod $p$ congruences of Fibonacci numbers}

\begin{document}

\begin{abstract}
In this paper, we develop an explicit method to express finite algebraic numbers 
(in particular, certain idempotents among them) in terms of linear recurrent sequences, 
and give applications to the characterization of the splitting primes in a given finite Galois extension over the rational field. 
\end{abstract}

\maketitle

%%%%%%%%%%%%%%%%%%%%%%%%%%%%%%%%%%%%%%%%
\section{Introduction}

%%%%%%%%%%%%%%%%%%%%%%%%%%%%%%%%%%%%%%%%
\subsection{Splitting of primes}
We denote by $\bP_\Q$ the set of all rational prime numbers. 
Let $L$ be a finite Galois extension over $\Q$ (contained in a fixed algebraic closure $\overline{\Q}$ of $\Q$) 
and $O_L$ the ring of integers in $L$.
%Every non-trivial ideal of $O_L$ is uniquely written as a product of prime ideals.
For $p\in \bP_\Q$, the ideal of $O_L$ generated by $p$ has a unique factorization into prime ideals: 
\begin{equation}\label{eq:prime_decom}
pO_L = \prod_{i=1}^r\mathfrak{p}_i^{e_i},
\end{equation} 
where $\mathfrak{p}_1,\ldots ,\mathfrak{p}_r$ are distinct prime ideals of $O_L$. %, called the prime ideal of $L$ lying above $p$ 
Since $L/\Q$ is Galois, we have $e_1=\cdots=e_r=:\e$ and 
$[O_L/\mathfrak{p}_1:\F_p]=\cdots=[O_L/\mathfrak{p}_r:\F_p]=:\f$, where $\F_p=\Z/p\Z$.
%We say that the primes $\mathfrak{p}_i$ lie above $p$ and write $\mathfrak{p}_i\mid p$. 
The prime number $p$ is said to \emph{split completely} in $L$ if $\e=1$ and $\f=1$.
%The prime ideal $\mathfrak{p}_i$ is called \emph{ramified} in $L$ if $e>1$; otherwise it is \emph{unramified}.
%There are only finitely many prime ideals which are ramified in $L$ (cf.~\cite{Neukirch}).
Determining the set 
\[\Spl(L):=\{p\in \bP_\Q \mid p\ \mbox{splits completely in $L$}\}\]
is an important theme in algebraic number theory; 
in fact, it is known that $L$ is determined uniquely by the set $\Spl(L)$. 
More precisely, 
if $M$ is another finite Galois extension over $\Q$ (in the same algebraic closure $\overline{\Q}$) 
and $\Spl(M)$ is equal to $\Spl(L)$ except for finitely many primes, then $L=M$ holds 
(see \cite[Theorem 8.19]{Cox}).

In this paper, we give a characterization of primes splitting completely in $L$ through the study of finite algebraic numbers introduced by the first author.
The basic question behind our study is as follows.
\begin{question}\label{main_question}
Given a finite Galois extension $L/\Q$, is there a rule which, for every prime $p$, determines whether $p$ belongs to $\Spl(L)$? 
\end{question}
%A good expository paper on  can be \cite{Weinstein} by Weinstein, which is a survey paper covering the range from Fermat to Scholtz' results.

This question is intimately related to the splitting of polynomials modulo primes studied in the context of a reciprocity law 
(see \cite{Wyman} for the exposition of a reciprocity law). 
In general, let $f(x)\in\Z[x]$ be a monic irreducible polynomial and 
$K=\Q(\alpha)\cong\Q[x]/(f(x))$ be the number field generated by a root $\alpha$ of $f$ 
(note that $K$ is not necessarily Galois over $\Q$). 
Then we have 
\begin{equation}\label{eq:splitting_poly} 
p\in \Spl(K) \Longleftrightarrow N_p(f)=\deg f
\end{equation}
for all but finitely many primes $p$ (cf.~\cite[Chapter I, Proposition 25]{Lang} and Remark after that), 
% Let $\alpha\in O_L$ be a generator of $L$ (i.e.~$L=\Q(\alpha)$) and $f(x)\in\Z[x]$ its minimal polynomial. 
% Then 
% \begin{equation}\label{eq:splitting_poly} 
% p\in \Spl(L) \Longleftrightarrow N_p(f)=\deg f
% \end{equation}
% holds for all but finitely many primes $p$ (cf.~\cite[Proposition 5.11]{Cox}), 
where we denote by $N_p(f)$ the number of distinct roots in $\bbF_p$ of the reduction of $f(x)$ modulo $p$. 
% This provides a description of the set $\Spl(L)$.
For example, the quadratic reciprocity law shows that $N_p(x^2-5)=2$ 
if and only if $p\equiv 1,4\bmod5$, which implies 
\begin{equation}\label{eq:Q5}
\Spl\big(\Q(\sqrt{5})\big)  =\{p\in \bP_\Q\mid p\equiv 1,4\bmod{5}\}.
\end{equation}

Similarly to the quadratic reciprocity law, the set $\Spl(L)$ is characterized by congruence conditions
if the Galois group $\Gal(L/\Q)$ is abelian. 
This is a consequence of the class field theory over $\Q$.
If $L$ is not abelian, even though the set $\Spl(L)$ is not characterized by congruence conditions, 
there are several studies that systematically provide explicit descriptions of $\Spl(L)$.
For example, let $L=\mathbb{Q}(\sqrt[3]{2}, e^{2\pi \sqrt{-1}/3})$ be the splitting field of $x^3-2$.
Its Galois group is isomorphic to the symmetric group $\mathfrak{S}_{3}$. 
Then, with finitely many exceptions, we have the following three equivalent conditions (see \cite[Theorem 1.1]{HS}):
\begin{equation}\label{eq:equivalent conditions x^3-2}
\begin{split}
p\in \Spl(L) &\Longleftrightarrow N_p(x^3-2)=3,\\
&\Longleftrightarrow \mbox{there exists $x,y\in\Z$ such that $p=x^2+27y^2$},\\
&\Longleftrightarrow b_p=2,
\end{split}
\end{equation}
where $\sum_{m\ge0}b_mq^m = q\prod_{n\ge1} (1-q^{6n})(1-q^{18n})=\eta(6\tau)\eta(18\tau)$ is a modular form of weight 1.
The first equivalence follows from \eqref{eq:splitting_poly} and the fact that $p$ splits completely in $K\cong\Q[x]/(x^3-2)$ if and only if $p$ splits completely in the Galois closure $L$ of $K$ (cf.~\cite[Chapter 1, \S8, Exercise 4]{Neukirch}). 
The second one comes from the study of primes of the form $x^2+n y^2$ 
(\cite[Theorem 9.4 and Proposition 9.5]{Cox}).
The last description of $ \Spl(L) $ is a special case of the Langlands program (cf.~\cite[\S5.1.2]{HS}): 
a correspondence between Galois representations and modular forms.
For more examples and recent developments on the reciprocity laws, see \cite{Weinstein} and references therein.
%, which contains a recent breakthrough by Scholze \cite{Scholze}.

%%%%%%%%%%%%%%%%%%%%%%%%%%%%%%%%%%%%%%%%
\subsection{Approaches via linear recurrent sequences}\label{subsec:approach}

Some of {the existing research} on congruences of linear recurrent sequences {has} been (often implicitly) linked with the decomposition laws of primes.
For example, let $(F_m)_m$ be the Fibonacci sequence, i.e., the sequence of integers satisfying the linear recurrence relation $F_m-F_{m-1}-F_{m-2}=0$ for $m\ge2$ with the initial values $(F_0,F_1)=(0,1)$.
It is known {(cf.~\cite[Chap.~XVII]{Dickson34})} that for any odd prime $p$ we have 
\[F_p\equiv 1\bmod p \iff p\equiv 1,4\bmod{5},\] 
which, by \eqref{eq:Q5}, gives another description of $\Spl\big(\Q(\sqrt{5})\big)$:
\[
\Spl\big(\Q(\sqrt{5})\big)  =\{p\in \bP_\Q\setminus\{2\} \mid F_p\equiv 1\bmod{5}\}.
\]
In this example, the quadratic field $\Q(\sqrt{5})$ naturally appears as the splitting field of the characteristic polynomial $x^2-x-1$ of the linear recurrent sequence $(F_m)_m$.

In a similar context, there are several studies on a rule determining $N_p(f)$ 
by a linear recurrent sequence modulo $p$ (cf.~\cite{MoreeNoubissie,Saito,Sun} and references therein).
As examples of solutions to Question \ref{main_question}, let us rephrase Sun's and Saito's results 
(note that these are not equivalent to the original statements 
from \cite[Theorem 5]{MoreeNoubissie} and \cite[Theorem 1]{Saito}).

\begin{theorem}\label{thm:SunSaito}
\begin{itemize}
\item[(1)] (Z.-W.~Sun) Let $f(x)=x^3+c_1x^2+c_2x+c_3\in \Z[x]$ be a monic irreducible polynomial such that $c_1^2\ne 3c_2$ and $L$ its splitting field over $\Q$.
Define the integer sequence $(s_m)_m$ by the linear recurrence 
$s_m+c_1s_{m-1}+c_2s_{m-2}+c_3s_{m-3}=0\ (\forall m\ge 3)$ with the initial values 
$(s_0,s_1,s_2)=(3,-c_1,c_1^2-2c_2)$.
Then,
\[ p\in \Spl(L) \Longleftrightarrow s_{p+1} \equiv c_1^2-2c_2 \bmod p\]
holds for all but finitely many primes $p$.
\item[(2)] (S.~Saito)
Let $f=x^d+c_1x^{d-1}+\cdots+c_d\in \Z[x]$ be a monic irreducible polynomial of degree $d\ge2$ 
with non-zero discriminant and $L$ its splitting field over $\Q$.
Define the integer sequence $(u_m)_m$ by $u_m+c_1u_{m-1}+\cdots+c_du_{m-d}=0\ (\forall m\ge d)$ with the initial values $(u_0,\ldots,u_{d-2},u_{d-1})=(0,\ldots,0,1)$.
Then,
\[ p\in \Spl(L) \Longleftrightarrow u_{p-2+d} \equiv 1 \bmod p\]
holds for all but finitely many primes $p$.
\end{itemize}
\end{theorem}

The condition that $c_1^2\ne 3c_2$ in Theorem \ref{thm:SunSaito} (1) is necessary 
to distinguish the cases $N_p(f)=1$ and $N_p(f)=3$.
Due to this, Theorem \ref{thm:SunSaito} (1) can not be applied to the case $f(x)=x^3-2$.
{On the other hand}, Theorem \ref{thm:SunSaito} (2) does cover the case.
Namely, we have that
\begin{equation}\label{eq:S_3_ext}
p \in \Spl\big(\mathbb{Q}(\sqrt[3]{2}, e^{2\pi \sqrt{-1}/3})\big) \Longleftrightarrow u_{p+1}\equiv 1\bmod p
\end{equation}
holds for all but finitely many primes $p$, where the integer sequence $(u_m)_m$ is defined by $u_m=2u_{m-3} \ (m\ge3)$ with $(u_0,u_{1},u_{2})=(0,0,1)$.
Notice that in \cite[Theorem 1]{Saito}, $d$ is assumed to be $\ge3$, but the case $d=2$ in Theorem \ref{thm:SunSaito} (2) also holds and covers the case of $f(x)=x^2-x-1$, the Fibonacci numbers $(F_m)_m$.

We remark that Theorem \ref{thm:SunSaito} is weaker than the full statements \cite[Theorem 5]{MoreeNoubissie} and \cite[Theorem 1]{Saito}; they also determined exceptional primes 
for which the claimed equivalences do not hold.
Later, we will come back to the determination problem of exceptional primes.

%%%%%%%%%%%%%%%%%%%%%%%%%%%%%%%%%%%%%%%%
\subsection{Our result}
The reason we write Theorem \ref{thm:SunSaito} in this manner is because it fits the framework of {\it finite algebraic numbers}.
The notion of finite algebraic number is introduced in \cite{Rosen20} as an $\mathcal{A}$-analogue of the period interpretation of the algebraic numbers (see Remark \ref{rem:motivic} for more details).
Here the symbol $\mathcal{A}$ stands for the ring $\left( \prod_{p} \F_p \right) \big/ \left( \bigoplus_{p} \F_p\right)$, where $p$ runs over all primes.
This ring was first introduced by Kontsevich \cite[\S2.2]{Kontsevich} and recently used in the study of multiple zeta values (see Remark \ref{rem:mzv}).
In this setting, the sequence $(a_p)_p$ is zero in $\mathcal{A}$ if $a_p=0$ for all but finitely many primes $p$.

This framework provides a conceptual explanation as to why the set $\Spl(L)$ can be characterized by the values of a linear recurrent sequence modulo primes.
Moreover, using the same machinery, we can give a law of decomposition of primes in $L$, namely, a classification of primes (which are unramified in $L/\Q$) according to the number of prime factors in \eqref{eq:prime_decom}.
Recall that $p$ is \emph{unramified} in $L/\Q$ if $\e=1$ in the decomposition \eqref{eq:prime_decom}.
In this case, the number $r=[L:\Q]/\f$ of factors is determined by the conjugacy class of $\Gal(L/\Q)$ 
to which the Frobenius automorphism at $p$ belongs. 
For the precise statement, let 
\begin{equation}\label{eq:rec}
\Rec(f;\Q) \coloneqq \bigg\{ (a_m)_m \in \prod_{m\ge0} \Q \ \bigg| \ a_m+c_1a_{m-1}+\cdots+ c_d a_{m-d} =0\ \mbox{for all}\  m\ge d\bigg\}
\end{equation}
be the $\Q$-vector space of sequences satisfying the homogeneous linear recurrence relation with characteristic polynomial $f(x)=x^d+c_1x^{d-1}+\cdots+c_d\in \Z[x]$ of degree $d$.
% the set $\{\sigma D_\frp \sigma^{-1} \mid \sigma\in \Gal(L/\Q)\}$ corresponds to, where $D_\frp$ is the decomposition group of $\frp$ in $L$ lying above $p$.

\begin{theorem}\label{thm:intro}
Let $L$ be a finite Galois extension over $\Q$ and $f\in \Q[x]$ a monic irreducible polynomial such that all roots of $f$ are simple and form a basis of $L$ over $\Q$.
Then, for any conjugacy class $C$ in the Galois group $\Gal(L/\Q)$, 
there exists a unique sequence $(a_m)_m=(a_m(C))_m\in \Rec(f;\Q)$ such that 
\begin{equation}\label{eq:decomposition_law_C} a_p\equiv \begin{cases} 
1 \mod p & (\text{if the Frobenius automorphism at $p$ belongs to $C$}),\\ 
0 \mod p & (\text{otherwise})
\end{cases}
\end{equation}
holds for all but finitely many primes $p$. 
In particular, letting $C=\{\id\}$, we have 
\begin{equation}\label{eq:splitting_completely}
a_p\equiv \begin{cases} 
1 \mod p & (\text{if $p\in \Spl(L)$}),\\ 
0 \mod p & (\text{otherwise}) 
\end{cases}
\end{equation}
for all but finitely many primes $p$. 
\end{theorem}

% The assumption on the polynomial $f$ in Theorem \ref{thm:intro} means that 
% $f$ is the minimal polynomial of a normal element of $L$ 
% (an element of $L$ is said {to be} normal if its Galois conjugates form a basis of $L$ over $\Q$). 
% In particular, this implies that $\deg f=[L:\Q]$. 
% {It is true that this assumption on $f$ is stronger than the one in Theorem \ref{thm:SunSaito} (2) and that Theorem \ref{thm:SunSaito} (2) already gives enough information about primes splitting completely.
% However, in Theorem \ref{thm:SunSaito} (2), $u_{p-2+d}\bmod p$ for $p\not\in \Spl(L)$ is in general neither determined nor constant. 
% The crucial difference with Theorem \ref{thm:intro} is that, with our assumption on $f$, we can control $a_p\bmod p$ for all but finitely many primes $p$, depending on to which conjugacy class $p$ belongs. 
% In other words, our assumption can be weakened if one wants to capture only primes splitting completely.}
% For example, despite the {fact that} $[\mathbb{Q}(\sqrt[3]{2}, e^{2\pi \sqrt{-1}/3}):\Q]=6$, 
% it follows from \eqref{eq:S_3_ext} (a special case of Theorem \ref{thm:SunSaito} (2)) that 
% the sequence $(a_m)_m\in {\rm Rec}(x^3-2;\Q)$ with $(a_0,a_1,a_2)=(0,1,0)$ satisfies 
% \[p \in \Spl\big(\mathbb{Q}(\sqrt[3]{2}, e^{2\pi \sqrt{-1}/3})\big) \Longleftrightarrow a_{p}\equiv 1\bmod p\]
% for all but finitely many primes $p$. 

The assumption on the polynomial $f$ in Theorem \ref{thm:intro} means that 
$f$ is the minimal polynomial of a normal element of $L$ 
(an element of $L$ is said {to be} normal if its Galois conjugates form a basis of $L$ over $\Q$). 
In particular, this implies that $\deg f=[L:\Q]$. 

\begin{remark}
{
Let us compare our Theorem \ref{thm:intro} with Saito's result \cite[Theorem 1]{Saito} 
of which Theorem \ref{thm:SunSaito} (2) is a special case. 
First, our assumption on the polynomial $f$ to be the minimal polynomial of a normal element 
is stronger than the one in Saito's result. 
For example, to apply Theorem \ref{thm:intro} to the field $\mathbb{Q}(\sqrt[3]{2}, e^{2\pi \sqrt{-1}/3})$, 
we need some polynomial of degree $6$. However, } 
it follows from \eqref{eq:S_3_ext} (a special case of Theorem \ref{thm:SunSaito} (2)) that 
the sequence $(a_m)_m\in {\rm Rec}(x^3-2;\Q)$ with $(a_0,a_1,a_2)=(0,1,0)$ satisfies 
\[p \in \Spl\big(\mathbb{Q}(\sqrt[3]{2}, e^{2\pi \sqrt{-1}/3})\big) \Longleftrightarrow a_{p}\equiv 1\bmod p\]
for all but finitely many primes $p$. 
{In this sense, Theorem \ref{thm:SunSaito} (2) is more generally applicable 
than our Theorem \ref{thm:intro}. }

{On the other hand, Theorem \ref{thm:intro} is finer than Saito's result:  
From the latter, one can read the information about the \emph{order} of the Frobenius automorphisms, 
while our result also gives the information of \emph{conjugacy classes} of them. 
We also remark that, with Saito's construction, one cannot control $a_p\bmod p$ for $p\notin\Spl(L)$ 
as done in \eqref{eq:splitting_completely}. }
\end{remark}

An interesting result is obtained when we apply Theorem \ref{thm:intro} to a characterization of 
the $p$-th Fourier coefficients of a modular form of weight 1 
corresponding to an irreducible odd Artin representation 
$\rho\colon\Gal(\overline{\Q}/\Q)\rightarrow {\rm GL}_2(\C)$ (cf.~\cite{DS74, Khare1997,KW091,KW092}).
If $\sum_{m\ge1}b_mq^m$ is such a modular form and $L$ is the fixed field by the kernel of $\rho$, 
then $b_p = \Tr(\rho(\phi_\frp))$ holds for any prime $p$ unramified in $L/\Q$  (note that there are only finitely many primes which are not unramified, i.e., ramify in $L/\Q$), 
where $\phi_\frp\in\Gal(L/\Q)$ denotes the Frobenius automorphism at a prime $\frp$ of $L$ above $p$.
Since the trace $\Tr(\rho(\sigma))$ depends only on the conjugacy class of $\sigma$ in $\Gal(L/\Q)$, 
the following linear combination of the sequences $(a_m(C))_m$ in Theorem \ref{thm:intro} is well-defined: 
\[a_m(\rho)=\sum_{\text{$C\subset\Gal(L/\Q)$: conjugacy class}} \Tr(\rho(\sigma_C))a_m(C) 
\qquad (\sigma_C\in C). \]
Then we have 
\[a_p(\rho)\equiv \Tr(\rho(\phi_\frp))\equiv b_p \bmod{p}\]
for all but finitely many primes $p$. 
The following is an example of this construction: 

\begin{proposition}\label{prop:S_3_intro}
Let $f(x)=x^{6}+3x^{5}+12x^{4}+25x^{3}+60x^{2}+51x+127$, whose splitting field is $\mathbb{Q}(\sqrt[3]{2}, e^{2\pi \sqrt{-1}/3})$.
Define the sequence $(a_m)_m\in {\rm Rec}(f;\Q)$ by the initial values $(a_0,a_1,\ldots,a_5)=(0,2,-4,-3,20,-40)$.
Then, the $p$-th Fourier coefficient of $\eta(6\tau)\eta(18\tau)=\sum_{m\ge0}b_mq^m$ 
as in \eqref{eq:equivalent conditions x^3-2} satisfies
\[ a_p\equiv b_p \mod p\]
for all but finitely many primes $p$. 
%Let $L=\mathbb{Q}(\sqrt[3]{2}, e^{2\pi \sqrt{-1}/3})$.
%In this case, the Galois group $\Gal(L/\Q)\cong\mathfrak{S}_3$ has a unique $2$-dimensional irreducible 
%representation $\rho$ over $\C$, which satisfies 
%\[ b_p = \Tr(\rho(\phi_\frp)) = \begin{cases} 2 & (|\langle \phi_\frp \rangle|=1), \\ 0 & (|\langle \phi_\frp \rangle|=2), \\ -1 & (|\langle \phi_\frp \rangle|=3),\end{cases}\]
%where $b_p$ is the $p$-th Fourier coefficient of $\eta(6\tau)\eta(18\tau)=\sum_{m\ge0}b_mq^m$ 
%as in \eqref{eq:equivalent conditions x^3-2}
%and $\langle \phi_\frp \rangle$ its conjugacy class.
%To obtain a sequence $(a_m)_m$, let $\xi=e^{2\pi\sqrt{-1}/3}(1+(\sqrt[3]{2})^{2})+\sqrt[3]{2}\in L$.
%This is a normal element 
%and its minimal polynomial is 
%\begin{align*}
%f(x)=x^{6}+3x^{5}+12x^{4}+25x^{3}+60x^{2}+51x+127.
%\end{align*}
%By Theorem \ref{thm:intro}, for each conjugacy class $C$ of ${\rm Gal}(L/\Q)$ we can find $(a_m(C))_m\in {\rm Rec}(f;\Q)$ satisfying \eqref{eq:a_p(C)}.
%Summing these up, one obtains $(a_m)_m\in {\rm Rec}(f;\Q)$ with the initial values $(a_0,a_1,\ldots,a_5)=(0,2,-4,-3,20,-40)$ such that
%\[ a_p\equiv b_p \mod p\]
%holds for all but finitely many primes $p$. 
\end{proposition}

Here we list the values of $a_p$ and $b_p$ for small primes $p$: 
\[\begin{array}{r|rrrrrrrrr}
p   &  2 &  3 &   5 &    7 &   11 &     13 &     17 & \dots &            31\\
\hline
a_p & -4 & -3 & -40 & -169 & 8690 & -49336 & 854726 & \dots & -647044186129\\
b_p &  0 &  0 &   0 &   -1 &    0 &     -1 &      0 & \dots &             2
\end{array}\]

Our proof of Theorem \ref{thm:intro} relies on the theory of finite algebraic numbers; for example, the fact that every finite algebraic number is obtained from a linear recurrent sequence over $\Q$ (see Theorem \ref{thm:r_f vs P_L^A}).
It is possible to generalize this theory to the relative case, i.e., 
the base field $\Q$ can be replaced by an arbitrary number field $K$, 
with no essential difficulty, but we will not develop this point here. 
The proof provides a method of 
explicit computation of the sequence describing the decomposition laws.

%%%%%%%%%%%%%%%%%%%%%%%%%%%%%%%%%%%%%%%%
\subsection{Refinement}
As we have mentioned in the end of \S\ref{subsec:approach}, there are finitely many exceptional primes for which the congruences in Theorem \ref{thm:intro} do not hold.
To determine these exceptions, we refine the theory of finite algebraic numbers 
working over the ring of $S$-integers $\Z_S:=\Z\left[ p^{-1} \mid p\in S\right]$ 
for a finite subset $S$ of $\bP_\Q$.
In this setting, we introduce the ring of $S$-integers in finite algebraic numbers and 
prove that every such $S$-integer is obtained from 
a linear recurrent sequence over $\Z_S$ (cf.~Theorem \ref{thm:S-refinement}).
With this refinement, we can control exceptional primes in Theorem \ref{thm:intro} as follows.

\begin{theorem}\label{thm:exceptional prime}
Let $L$, $f$ and $(a_m)_m=(a_m(C))_m\in \Rec(f;\Q)$ be as in Theorem \ref{thm:intro} and denote all distinct roots of $f$ by $\xi_1,\ldots,\xi_d$, where {$d=[L:\Q]$}.
Suppose that a finite subset $S\subset \bP_\Q$ of primes satisfies the following conditions:
\begin{enumerate}
\item[(S0)] Each component of the sequence $(a_m)_m$ is an $S$-integer. %$S$ contains all primes which divide denominators of the initial values $(a_0,\ldots,a_{m-1})$.
\item[(S1)] $S$ contains all primes ramifying in $L/\Q$. 
\item[(S2)] $f$ belongs to $\Z_S[x]$ and $\disc(f):=\prod_{i<j} (\xi_i-\xi_j)^2$ is invertible in $\Z_S$. 
%\item[(S3)] The determinant of the $d\times d$ matrix $\bigl(\Tr(\xi_i\xi_j)\bigr)_{i,j}$ is invertible in $\Z_S$.
\end{enumerate}
Then, the congruence \eqref{eq:decomposition_law_C} holds for all primes $p\not\in S$. 
\end{theorem}

Note that, if $f\in \Z_S[x]$, then we can confirm the condition (S0) by checking $a_m\in\Z_S$ 
for $m=0,\ldots,d-1$. 

Let us illustrate an example of the use of Theorem \ref{thm:exceptional prime}.

\begin{example}\label{ex:Q8}
Set $\omega_1=\sqrt{(2+\sqrt{2})(3+\sqrt{3})}$ and let $L=\mathbb{Q}(\omega_1)$.
Then its Galois group is isomorphic to the quaternion group $Q_{8}$.
The element $\xi=(1+\sqrt{2})(1+\sqrt{3})(1+\omega_1)$ is a normal element of $L$
and its minimal polynomial is 
\[f(x)=x^{8}-8x^{7}-736x^{6}-3344x^{5}+5800x^{4}+18272x^{3}-27904x^{2}+9920x-368. \]
By Theorem \ref{thm:intro}, one can find the sequence $(a_m)_m\in \Rec(f;\Q)$ such that \eqref{eq:splitting_completely} holds for all but finitely many primes $p$.
Using the method described in \S\ref{sec:method}, we obtain the initial values of $(a_m)_m$ as follows:
\begin{align*}
&
a_{0}=\frac{1}{8}, \quad a_{1}=1, \quad a_{2}=\frac{1393}{20}, \quad
a_{3}=\frac{3199}{2}, \\
&
a_{4}=\frac{2003629}{30}, \quad
a_{5}=1936618, \quad
a_{6}=\frac{1043676173}{15}, \quad
a_{7}=\frac{10973964638}{5}.
\end{align*}
Now let us determine exceptional primes.
For the first condition (S0), we count prime factors of denominators of the initial values $(a_0,\ldots,a_7)$ which are $2$, $3$ and $5$.
For (S1), we recall that a prime $p$ ramifies in $L/\Q$ if (and only if) $p$ divides the discriminant $\disc (L)$ of $L$ (see \eqref{eq:def_disc_L} for the definition).
The prime factorizations of $\disc(L)$ and $\disc(f)$
%$\disc(L)$, $\disc(f)$ and $\det\bigl(\Tr(\xi_i\xi_j)\bigr)_{i,j}$ 
are given by  
\begin{align*}
\disc(L)&=2^{24} \cdot 3^{6},\\
\disc(f)&=2^{72}\cdot 3^6\cdot 23^2 \cdot 241^2\cdot 359^2\cdot 147409^2.%,\\
%\det\bigl(\Tr(\xi_i\xi_j)\bigr)_{i,j} &= 2^{64}\cdot 3^8 \cdot 11 \cdot 23.
\end{align*}
Thus, the set 
\[S=\{2,3,5,23,241,359,147409\} \]
satisfies the conditions (S0) to (S2) in Theorem \ref{thm:exceptional prime}.
Hence, the above sequence $(a_m)_m$ satisfies \eqref{eq:splitting_completely}
%\[a_p\equiv\begin{cases}
%1 \mod p & (\text{if $p\in \Spl(L)$}), \\
%0 \mod p & (\text{otherwise})
%\end{cases}\]
for all prime $p\notin S$.
\end{example}

To complement the list of prime{s} splitting completely in $L$, we may use a mathematical software.
Indeed, we can confirm that every $p\in S$ does not split completely in $L$ of Example \ref{ex:Q8}.
This can be checked by the following code in SAGEMATH \cite{sage}:\\
\verb|sage: K.<y> = NumberField(x^8 - 8*x^7 - 736*x^6 - 3344*x^5 + 5800*x^4 + 18272*x^3 |\\
\verb|- 27904*x^2 + 9920*x - 368)|\\
\verb|sage: K.completely_split_primes(147409)|\\
\verb|[71, 191, 239, 313, 337, 383, ..., 147311]|

\begin{remark}
In the same way, we obtain a set $S$ of (possibly) exceptional primes in Proposition \ref{prop:S_3_intro} 
as follows (note that in this case, we have $\disc(L)=-2^4\cdot 3^7$ and $\disc(f)=-2^4\cdot 3^{17}\cdot 5^2\cdot 11^2$).
\[S=\{2,3,5,11\}. \]
By a direct calculation, one can check that these four primes are actually not exceptional 
(see the table given after Proposition \ref{prop:S_3_intro}). 
Therefore, in this case, we conclude that the congruence 
$a_p\equiv b_p\pmod{p}$ holds for \emph{every} prime $p$.  
\end{remark}

\subsection{Contents}
The paper is structured as follows.
In \S2, we recall the theory of finite algebraic numbers developed by the first author in \cite{Rosen20}, 
using a slightly general convention. 
In \S3, we explain the method to find a recurrent sequence $(a_m)_m$ of Theorem \ref{thm:intro} and {illustrate} it in the cases of $[L:\Q]=2$ and $3$. 
At this stage, there is a problem that our method described in \S3.2 does not provide any information about the finite set $S$ of exceptional primes.
In \S4, we settle this problem by showing Theorem \ref{thm:exceptional prime} and give some more examples %of Theorem \ref{thm:intro} 
in which both $S$ and $(a_m)_m$ are explicitly determined.

Notation: Throughout the paper, we let $L$ be a finite Galois extension of $\Q$ and set $G=\Gal(L/\Q)$. 

%%%%%%%%%%%%%%%%%%%%%%%%%%%%%%%%%%%%%%%%
\section{Finite Algebraic Numbers}
In this section, we review some basic constructions relevant to finite algebraic numbers.
Our expositions are rendered as self-contained as possible, whilst the materials and some of results presented here are essentially contained in \cite{Rosen20}. 

%%%%%%%%%%%%%%%%%%%%%%%%%%%%%%%%%%%%%%%%
\subsection{Definition of Finite Algebraic Numbers}

%Let us define the finite algebraic number.
For a prime $p$, let $\F_p$ be the finite field of order $p$ and consider the ring 
\[\mathcal{A}\coloneqq\bigg( \prod_{p\in \bP_\Q} \F_p \bigg) \big/ \bigg( \bigoplus_{p\in \bP_\Q} \F_p\bigg).\]
An element of $\cA$ is denoted as $(a_p)_p$ with $a_p\in\F_p$ for each prime $p$, 
and two elements $(a_p)_p ,(b_p)_p \in \mathcal{A}$ are equal if $a_p=b_p$ in $\F_p$ holds for all but finitely many primes $p$. 
Thus, to describe an element $(a_p)_p\in\cA$, we can ignore a finite number of primes $p$. 

As an example, let $c$ be a rational number. 
Since we have $c\in \Z_{(p)}$ except for a finite number of primes $p$, 
where $\Z_{(p)}$ denotes the localization of $\Z$ at the prime ideal $(p)$, 
we can define an element $(c\bmod p)_p\in\cA$. 
Here, by an abuse of notation, we write ``$c\bmod p$'' for the class of $c\in \Z_{(p)}$ 
in the residue field $\Z_{(p)}/p\Z_{(p)}\cong \F_p$. 
Thus we obtain an injective ring homomorphism $\Q\to\cA$ given by $c\mapsto (c\bmod p)_p$, 
through which $\cA$ is viewed as a $\Q$-algebra. 

\begin{definition}\label{def:FAN}
An element $\alpha\in\cA$ is called a \emph{finite algebraic number} if 
there exists a sequence $(a_m)_m$ of rational numbers satisfying a homogeneous 
linear recurrence relation such that $\alpha=(a_p\bmod p)_p$. 
The set of all finite algebraic numbers is denoted by $\cP^\cA_0$. 
\end{definition}

As in \eqref{eq:rec}, let $\Rec(f;\Q)$ denote the $\Q$-vector space of rational sequences satisfying a homogeneous linear recurrence relation with characteristic polynomial $f(x)=x^d+c_1x^{d-1}+\cdots+c_d\in \Q[x]$ of degree $d$.
Since elements in $\Rec(f;\Q)$ are determined by initial values, it holds that $\dim_\Q \Rec(f;\Q)=d$.
For any sequence $(a_m)_m\in\Rec(f;\Q)$ and a prime $p$ which does not divide 
the denominators of the initial terms $a_0,\ldots,a_{d-1}$ and the denominators of 
the coefficients $c_1,\ldots,c_d$ of $f$, we see that all terms $a_m$ belong to $\Z_{(p)}$. 
In particular, $(a_p\bmod p)_p\in\cA$ is well defined. 
Therefore, there is a $\Q$-linear map 
\begin{equation}\label{eq:r_f}
r_f\colon \Rec(f;\Q)\longrightarrow \cA;\ 
(a_m)_m\longmapsto (a_p\bmod p)_p. 
\end{equation}
By definition, the subset $\cP^\cA_0$ of $\cA$ is the union of 
the images of these maps $r_f$, 
where $f$ runs over all monic polynomials over $\Q$. 
%Note that an element $(a_m)_m\in\Rec(f;\Q)$ is determined by the first $d$ entries 
%$a_0,\ldots,a_{d-1}\in\Q$, because the $\Q$-linear map $\Rec(f;\Q)\rightarrow \Q^d$ sending $(a_m)_m$ to $(a_0,\ldots,a_{d-1})$ is an isomorphism.

\begin{example}
\begin{enumerate}
\item 
A rational number, viewed as an element of $\cA$, is a finite algebraic number. 
Indeed, for $f(x)=x-1$, the image of the map $r_f\colon\Rec(f;\Q)\to\cA$ is 
exactly the rational numbers. 

\item 
Let $f(x)=x^2-x-1$.
The Fibonacci sequence $(F_m)_m$ given by $F_m=F_{m-1}+F_{m-2}$ with $F_0=0,F_1=1$ is 
an element of ${\rm Rec}(f;\Q)$, so $r_f((F_m)_m)$ is a finite algebraic number.
Remark that the Fibonacci numbers satisfy $F_p \equiv \big(\tfrac{5}{p}\big)\mod p$ for all primes $p$ (see \cite[Chap.~XVII]{Dickson34}), where $\big(\tfrac{D}{p}\big)$ denotes the Kronecker symbol.
Hence, 
\begin{equation*}\label{eq:Fibonacci} 
r_f((F_m)_m)=\left(\left(\frac{5}{p}\right)\right)_p.% =:\left(\frac{5}{\boldsymbol{p}}\right).
\end{equation*}
This shows that $\big(\big(\tfrac{5}{p}\big)\big)_p\in \mathcal{P}_0^{\mathcal{A}}$.
In this way, we can also prove that for any square-free integer $D$, 
the element $\big(\big(\tfrac{D}{p}\big)\big)_p\in \mathcal{A}$ is a finite algebraic number (see, for example, \cite[p.46]{R89}).
\end{enumerate}
\end{example}

The notion of finite algebraic number is first introduced by the first author \cite[Theorem 1.1]{Rosen20}, where three equivalent constructions are made.
The first construction is given in Definition \ref{def:FAN}. 
The second construction is Galois theoretic; 
we recast this in the following subsections 
and prove its equivalence with the first construction (see Corollary \ref{cor:cP^cA_0 and L}).
For the third construction, see Remark \ref{rem:motivic}.

%%%%%%%%%%%%%%%%%%%%%%%%%%%%%%%%%%%%%%%%
\subsection{The Ring $\cA_L$}
Recall that $L$ denotes a finite Galois extension of $\Q$. 
Let $\bP_L$ denote the set of all maximal ideals of $O_L$, the ring of integers in $L$. 
Then we define the ring $\cA_L$ associated with $L$ by 
\[\cA_L\coloneqq \Biggl(\prod_{\frp\in \bP_L}O_L/\frp\Biggr)
\biggm/\Biggl(\bigoplus_{\frp\in \bP_L}O_L/\frp\Biggr). \]
In particular, the ring $\cA_\Q$ coincides with the ring 
$\cA=\bigl(\prod_p\bbF_p\bigr)/\bigl(\bigoplus_p\bbF_p\bigr)$ introduced in the previous subsection, 
where we identify the set $\bP_\Q$ with the set of rational primes. 

As in the case of $\cA_\Q=\cA$, 
an element $(a_\frp)_\frp$ of $\cA_L$ is specified by giving elements $a_\frp\in O_L/\frp$ 
for all but finitely many $\frp\in\bP_L$. 
In particular, for any $c\in L$, we have a well-defined element $(c\bmod\frp)_\frp\in\cA_L$ 
since $c\in O_{L,\frp}$ (the localization of $O_L$ at $\frp$) 
for all but finitely many $\frp\in\bP_L$. 
Thus $\cA_L$ admits an injective ring homomorphism $L\to\cA_L$ given by 
$c\mapsto (c\bmod\frp)_\frp$, through which $\cA_L$ is viewed as an $L$-algebra. 

% As another example, let us consider the idempotents in the ring $\cA_L$. 
% For any subset $\Sigma$ of $\bP_L$, let $I_\Sigma\colon\bP_L\to\{0,1\}$ denote its indicator, i.e., 
% \[I_\Sigma(\frp)=\begin{cases}
% 1 & (\frp\in \Sigma),\\
% 0 & (\frp\notin \Sigma). 
% \end{cases}\]
% Then the element $e^\cA_\Sigma\coloneqq \bigl(I_\Sigma(\frp)\bigr)_\frp\in\cA_L$ is an idempotent. 
% An arbitrary idempotent is equal to $e^\cA_\Sigma$ for some subset $\Sigma\subset\bP_L$, 
% and $e^\cA_{\Sigma_1}=e^\cA_{\Sigma_2}$ holds for two subsets $\Sigma_1,\Sigma_2\subset\bP_L$ 
% if and only if the symmetric difference of $\Sigma_1$ and $\Sigma_2$ is finite. 

We define an action of the Galois group $G=\Gal(L/\Q)$ on the ring $\cA_L$ as follows:
%The ring $\cA_L$ also has a natural $G$-action defined by 
\[\sigma\bigl((a_\frp)_\frp\bigr)=\bigl(\sigma(a_{\sigma^{-1}(\frp)})\bigr)_\frp
\quad\text{for $\sigma\in G$}. \]
Here the isomorphism $O_L/\sigma^{-1}(\frp)\to O_L/\frp$ induced by $\sigma\colon O_L\to O_L$ 
is denoted by the same symbol $\sigma$. 
The inclusion $L\hookrightarrow\cA_L$ is $G$-equivariant under this action. 

%%%%%%%%%%%%%%%%%%%%%%%%%%%%%%%%%%%%%%%%
\subsection{The Frobenius-Evaluation Map}
Let  
\[\Fun(G,L)\coloneqq\{g\colon G\to L\}\]
be the $L$-algebra of all functions from $G$ to $L$ 
with pointwise operations, e.g., $(g_1+g_2)(\tau)=g_1(\tau)+g_2(\tau)$. 
This ring also has a $G$-action defined by 
\[(\sigma g)(\tau)\coloneqq\sigma\bigl(g(\sigma^{-1}\tau\sigma)\bigr)
\quad \text{for $g\in\Fun(G,L)$ and $\sigma,\tau\in G$}, \]
under which the structure map $L\to\Fun(G,L)$ is $G$-equivariant. 

As is well-known in Galois theory, there is an isomorphism of $L$-algebras 
\[\varphi\colon L\otimes_\Q L\longrightarrow \Fun(G,L);\ 
\xi\otimes\eta\longmapsto \bigl(\tau\mapsto \xi\tau(\eta)\bigr), \]
where $L\otimes_\Q L$ is regarded as an $L$-algebra by the map $\xi\mapsto\xi\otimes 1$. 
This isomorphism is also $G$-equivariant if we let $G$ act on $L\otimes_\Q L$ diagonally: indeed, 
\[\varphi\bigl(\sigma(\xi\otimes\eta)\bigr)(\tau)=\sigma(\xi)\tau\sigma(\eta)
=\sigma\bigl(\xi\,\sigma^{-1}\tau\sigma(\eta)\bigr)
=\bigl(\sigma\varphi(\xi\otimes\eta)\bigr)(\tau). \]

If $\frp\in\bP_L$ is unramified in $L/\Q$, let $\phi_\frp\in G$ be the Frobenius automorphism. 

\begin{definition}\label{defn:ev}
We define the $L$-algebra homomorphism $\ev\colon\Fun(G,L)\to\cA_L$ by 
\[\ev(g)\coloneqq \bigl(g(\phi_\frp)\bmod\frp\bigr)_\frp. \]
This expression makes sense in $\cA_L$ since, for all but finitely many $\frp\in\bP_L$, 
$\frp$ is unramified in $L/\Q$ and $g(\tau)\in O_{L,\frp}$ for all $\tau\in G$. 
\end{definition}

\begin{proposition}\label{prop:ev}
The map $\ev\colon\Fun(G,L)\to\cA_L$ is $G$-equivariant and injective. 
\end{proposition}
\begin{proof}
The $G$-equivariance is directly verified as 
\[\ev\bigl(\sigma g)=\bigl(\sigma(g(\sigma^{-1}\phi_\frp\sigma))\bigr)_\frp
=\bigl(\sigma(g(\phi_{\sigma^{-1}(\frp)}))\bigr)_\frp=\sigma\bigl(\ev(g)\bigr). \]
To prove the injectivity, 
suppose that $g\in\Fun(G,L)$ satisfies $\ev(g)=0$.
Then, for any $\tau \in G$ there are infinitely many $\frp$ such that $g(\phi_\frp)\equiv 0 \bmod\frp$ and $\phi_\frp=\tau$ (because of the \v{C}ebotarev density theorem; see e.g.~\cite[VII (13.4)]{Neukirch}). 
This shows that $g(\tau)=0$ for all $\tau \in G$, as desired.
\end{proof}

%%%%%%%%%%%%%%%%%%%%%%%%%%%%%%%%%%%%%%%%
\subsection{The Space of Linear Recurrent Sequences}
Let $f(x)=x^d+c_1x^{d-1}+\cdots+c_d\in \Q[x]$ be a monic polynomial over $\Q$ of degree $d$. 
For any extension field $K$ of $\Q$, 
we let $\Rec(f;K)$ be the $K$-vector space of homogeneous linear recurrent sequences 
with characteristic polynomial $f$: 
\[\Rec(f;K)\coloneqq 
\Biggl\{(a_m)_m\in\prod_{m\ge 0}K\Biggm| a_m+c_1 a_{m-1}+\cdots+c_d a_{m-d}=0 
\text{ for } \forall m\ge d\Biggr\}. \]
The map $\Rec(f;K)\to K^d$ given by $(a_m)_m\mapsto (a_0,\ldots,a_{d-1})$ 
is an isomorphism of $K$-vector spaces. 

In the following, assume that $f$ decomposes into linear factors over $L$ 
(which is a finite Galois extension over $\Q$, as before). 
Let $\xi_1,\ldots,\xi_\nu\in L$ be all the distinct roots of $f$ and 
let $\mu_j$ denote the multiplicity of the root $\xi_j$, namely, $f(x)=\prod_{j=1}^\nu (x-\xi_j)^{{\mu}_j}$.
Then the sequences 
\begin{equation}\label{eq:basis of Rec(f;L)}
\bigl(\tbinom{m}{i}\xi_j^{m-i}\bigr)_m \quad 
\text{for $j=1,\ldots,\nu$ and $i=0,\ldots,\mu_j-1$}
\end{equation}
form a basis of the $L$-vector space $\Rec(f;L)$ 
(if $\xi_j=0$, we replace the above sequence with $(\delta_{m,i})_m$). 

Recall that $L\otimes_\Q L$ is equipped with an $L$-algebra (hence $L$-linear) structure 
given by $\xi\mapsto \xi\otimes 1$, and the diagonal $G$-action. 
On the other hand, the $L$-vector space $\Rec(f;L)$ also has a $G$-action defined entrywise: 
$\sigma\bigl((a_m)_m\bigr)\coloneqq\bigl(\sigma(a_m)\bigr)_m$. 
Note that these $G$-actions are semilinear over $L$, i.e., 
we have $\sigma(\xi v)=\sigma(\xi)\sigma(v)$ for $\sigma\in G$, 
$\xi\in L$ and $v\in L\otimes_\Q L$ or $\Rec(f;L)$. 

\begin{proposition}
Let $\psi_f\colon\Rec(f;L)\to L\otimes_\Q L$ be the $L$-linear map defined 
on the basis \eqref{eq:basis of Rec(f;L)} by 
\[\psi_f\Bigl(\bigl(\tbinom{m}{i}\xi_j^{m-i}\bigr)_m\Bigr)\coloneqq 
\begin{cases}
1\otimes \xi_j & (i=0), \\
0 & (i>0). 
\end{cases} \]
Then $\psi_f$ is a $G$-equivariant map. 
\end{proposition}
\begin{proof}
By semilinearity, it suffices to check the equivariance on the basis. 
For $j=1,\ldots,\nu$, $i=0,\ldots,\mu_j-1$ and $\sigma\in G$, 
we have 
\begin{align*}
\psi_f\Bigl(\sigma\bigl(\tbinom{m}{i}\xi_j^m\bigr)_m\Bigr)
=\psi_f\Bigl(\bigl(\tbinom{m}{i}\sigma(\xi_j)^m\bigr)_m\Bigr)
&=\begin{cases}
1\otimes\sigma(\xi_j) & (i=0), \\
0 & (i>0)
\end{cases}\\
&=\sigma\Bigl(\psi_f\Bigl(\bigl(\tbinom{m}{i}\xi_j^m\bigr)_m\Bigr)\Bigr), 
\end{align*}
where the second equality follows from that $\sigma(\xi_j)=\xi_{j'}$ for some $j'$ 
and $\mu_j=\mu_{j'}$. 
This shows the desired equivariance. 
\end{proof}

The map $\psi_f$ is not an isomorphism in general. 
In fact, we have the following: 

\begin{proposition}\label{prop: condition on f}
%For a monic polynomial $f\in\Q[x]$ with the distinct roots $\xi_1,\ldots,\xi_\nu\in L$, 
For a monic polynomial $f\in\Q[x]$ which decomposes into linear factors over $L$,
the following conditions are equivalent: 
\begin{enumerate}
\item The $L$-linear map $\psi_f\colon\Rec(f;L)\to L\otimes_\Q L$ is an isomorphism. 
\item %The roots $\xi_1,\ldots,\xi_\nu$ of
All roots of $f$ are simple, and form a basis of $L$ over $\Q$. 
\item $f$ is the minimal polynomial of a normal element of $L$ over $\Q$ 
(recall that $\xi\in L$ is called normal if its Galois conjugates $\sigma(\xi)$ ($\sigma\in G$) 
form a basis of $L$ over $\Q$).  
\end{enumerate}
In particular, for any finite Galois extension $L$ over $\Q$, 
there exists $f$ satisfying these conditions. 
\end{proposition}
\begin{proof}
The equivalence of the conditions (1) and (2) is clear from the construction of $\psi_f$. 

The condition (3) obviously implies (2). 
To show the converse, let us assume the condition (2). 
Then it is enough to prove that the $G$-action on the set $\{\xi_1,\ldots,\xi_\nu\}$ of all the distinct roots of $f$ is transitive. 
Note first that the sum of the elements of each $G$-orbit in $\{\xi_1,\ldots,\xi_\nu\}$ 
is a rational number, which is nonzero because of the linear independence. 
If there are two (or more) $G$-orbits, there exist two ways to express the rational numbers 
as linear combinations of $\xi_1,\ldots,\xi_\nu$, which contradicts the linear independence. 
Thus the set $\{\xi_1,\ldots,\xi_\nu\}$ is a single $G$-orbit, as desired. 

The ``in particular'' part follows from the normal basis theorem. 
\end{proof}

%%%%%%%%%%%%%%%%%%%%%%%%%%%%%%%%%%%%%%%%
\subsection{Taking Invariant Parts}
In the previous subsections, for a monic polynomial $f\in \Q[x]$ which decomposes into linear factors over $L$, we have constructed $G$-equivariant maps 
\[\Rec(f;L)\overset{\psi_f}{\longrightarrow} L\otimes_\Q L
\overset{\varphi}{\longrightarrow} \Fun(G,L)
\overset{\ev}{\longrightarrow} \cA_L. \]
Recall that the map $\varphi$ is bijective and that the map ${\rm ev}$ is injective (Proposition \ref{prop:ev}).
By taking the subspaces of $G$-invariant elements, we obtain the $\Q$-linear maps 
\begin{equation}\label{eq:G-invariant}
\Rec(f;\Q)\overset{\psi_f}{\longrightarrow} (L\otimes_\Q L)^G
\overset{\varphi}{\longrightarrow} A(L)
\overset{\ev}{\longrightarrow} \cA_L^G, 
\end{equation}
where we set 
\[A(L)\coloneqq \Fun(G,L)^G=\{g\colon G\to L\mid \sigma(g(\tau))=g(\sigma\tau\sigma^{-1})
\text{ for }\forall \sigma,\tau\in G\}. \]

We show that $\cA_L^G$ is {isomorphic} to $\cA=\cA_\Q$. 
First, we consider the diagonal embedding 
$\bbF_p\to \prod_{\frp\mid p}O_L/\frp$ for each rational prime $p$, 
where $\frp$ runs over the primes of $L$ lying above $p$. 
Then the product of these embeddings 
\[\prod_{p\in\bP_\Q}\bbF_p\longrightarrow 
\prod_{p\in\bP_\Q}\Biggl(\prod_{\frp\mid p}O_L/\frp\Biggr)=\prod_{\frp\in\bP_L}O_L/\frp\]
defines, by passage to the quotient, a $\Q$-algebra homomorphism $\cA\to\cA_L$. 

\begin{proposition}\label{prop:cA_L^G}
The map $\cA\to\cA_L$ constructed above 
induces an isomorphism $\cA\cong\cA_L^G$ of $\Q$-algebras. 
\end{proposition}
\begin{proof}
The injectivity of the map $\cA\to\cA_L$ follows from the equality 
\[\bigoplus_{p\in\bP_\Q}\Biggl(\prod_{\frp\mid p}O_L/\frp\Biggr)
=\bigoplus_{\frp\in\bP_L}O_L/\frp. \]
Moreover, its image is obviously contained in $\cA_L^G$. 
Conversely, let $(a_\frp)_\frp$ be an arbitrary element of $\cA_L^G$. 
This means that we have $\sigma(a_\frp)=a_{\sigma(\frp)}$ for all $\sigma\in G$ 
and all but finitely many $\frp\in\bP_L$. By replacing $a_\frp$ for {finitely many} $\frp$, 
we may assume that the above identity holds for all $\frp$. 
Then, for each $p\in\bP_\Q$, choose some $\frp\in\bP_L$ lying above $p$ and 
set $b_p\coloneqq a_\frp$. The $G$-invariance implies that 
this element $b_p\in O_L/\frp$ actually belongs to the subfield $\bbF_p$, 
and is independent of the choice of $\frp$. Thus we obtain an element $(b_p)_p\in\cA$, 
and by construction, this maps to the given $(a_\frp)_\frp\in\cA_L^G$. 
\end{proof}

In what follows, we identify $\cA_L^G$ with $\cA$ via the isomorphism 
in Proposition \ref{prop:cA_L^G}. 
%An analogue of $L$ in $\mathcal{A}$ is given as follows (see Remark \ref{rem:motivic}).

\begin{definition}\label{def:cP^cA_L}
For each finite Galois extension $L$ over $\Q$, 
we define a $\Q$-subalgebra $\cP^\cA_L$ of $\cA$ by 
\[\cP^\cA_L\coloneqq \ev\bigl(A(L)\bigr). \]
\end{definition}

We now relate $r_f\bigl(\Rec(f;\Q))$ with $\mathcal{P}_L^{\mathcal{A}}$, where the map $r_f$ is defined in \eqref{eq:r_f}.

\begin{theorem}\label{thm:r_f vs P_L^A}
Let $f\in \Q[x]$ be a monic polynomial which decomposes into linear factors over $L$.
\begin{enumerate}
\item 
The map $r_f\colon\Rec(f;\Q)\to\cA$ is equal to $\ev\circ\varphi\circ\psi_f$, 
the composition of the maps in \eqref{eq:G-invariant}. 
In particular, we have $r_f\bigl(\Rec(f;\Q)\bigr)\subset\mathcal{P}_L^{\mathcal{A}}$.
\item The map $r_f:\Rec(f;\Q)\rightarrow  \mathcal{P}_L^{\mathcal{A}}$ is an isomorphism of $\Q$-vector spaces if $f$ is the minimal polynomial of a normal element of $L$.
%For any finite Galois extension $L$ over $\Q$, there exists a monic polynomial $f\in\Q[x]$ such that $r_f\bigl(\Rec(f;\Q)\bigr)= \mathcal{P}_L^{\mathcal{A}}$, 
%and the equality $\dim_\Q \mathcal{P}_L^{\mathcal{A}}=[L:\Q]$ holds. 
\end{enumerate}
\end{theorem}
\begin{proof}
(1) Let $\xi_1,\ldots,\xi_\nu\in L$ be all the distinct roots of $f$.
For a sequence $(a_m)_m\in \Rec(f;\Q)$, if we write $a_m=\sum_{i,j}z_{i,j}\tbinom{m}{i}\xi_j^{m-i}$ with some $z_{i,j}\in L$ 
by using the basis \eqref{eq:basis of Rec(f;L)}, 
then for all but finitely many primes $p$, we have 
\[a_p=\sum_{i,j}z_{i,j}\tbinom{p}{i}\xi_j^{p-i}\equiv \sum_j z_{0,j}\xi_j^p
\equiv \sum_j z_{0,j}\phi_\frp(\xi_j) \mod\frp\]
with $\frp\in\bP_L$ lying above $p$. 
This shows the equality $r_f=\ev \circ \varphi\circ \psi_f$, 
and hence $r_f\bigl(\Rec(f;\Q)\bigr)\subset\ev\bigl(A(L)\bigr)=\mathcal{P}_L^{\mathcal{A}}$.

(2) If $f$ is the minimal polynomial of a normal element of $L$, then the $G$-equivariant map $\psi_f$ is an isomorphism by Proposition \ref{prop: condition on f}.
Hence the statement follows from (1). 
%Hence, by (1) shown above, $r_f\colon \Rec(f;\Q)\to\cP^\cA_L$ is also an isomorphism of 
%$\Q$-vector spaces. 
%In particular, we obtain $\dim_\Q \mathcal{P}_L^{\cA}=\dim_\Q \Rec(f;\Q)=\deg f= [L:\Q]$.
\end{proof}

%As an immediate consequence of Theorem \ref{thm:r_f vs P_L^A}, we have the following: 

\begin{corollary}\label{cor:cP^cA_0 and L}
\begin{itemize}
\item[(1)] We have that $\dim_\Q \mathcal{P}_L^{\mathcal{A}}=[L:\Q]$.
\item[(2)] The subset $\cP^\cA_0$ of $\cA$ is equal to the union of $\cP^\cA_L$, 
where $L$ runs over all finite Galois extensions of $\Q$. 
\end{itemize}
\end{corollary}
\begin{proof}
(1) Suppose that $f$ is the minimal polynomial of a normal element of $L$.
Then, by Theorem \ref{thm:r_f vs P_L^A} (2), we have $\dim_\Q \mathcal{P}_L^{\cA}=\dim_\Q \Rec(f;\Q)=\deg f= [L:\Q]$.

(2) The set $\cP^\cA_0$ is, by definition, the union of the images of maps $r_f$.
Therefore, the result follows from the ``in particular" part of Proposition \ref{prop: condition on f} and Theorem \ref{thm:r_f vs P_L^A}.
\end{proof}

Corollary \ref{cor:cP^cA_0 and L} (2) says that an element $\alpha$ of $\cA$ is a finite algebraic number if and only if there exists a finite Galois extension $L/\Q$ and $g\in A(L)$ such that ${\rm ev}(g)=\alpha$.
This was first proved in \cite[Theorem 2.2]{Rosen20}.

%%%%%%%%%%%%%%%%%%%%%%%%%%%%%%%%%%%%%%%%
\subsection{Ring of Finite Algebraic Numbers} 
Although not necessary for the proof of the main results, we clarify the structure of $A(L)$ in order to highlight the difference with $L$.
%We describe the structure of the $\Q$-algebra $A(L)$. 

\begin{proposition}\label{prop:product decomposition}
Let $R\subset G$ be a set of representatives from all conjugacy classes of $G$ and, 
for each $\varrho\in R$, let $C_G(\varrho)\coloneqq\{\sigma\in G\mid\sigma\varrho=\varrho\sigma\}$ 
be the centralizer of $\varrho$ in $G$. Then we have an isomorphism of $\Q$-algebras 
\[A(L)\cong \prod_{\varrho\in R}L^{C_G(\varrho)}. \]
In particular, the equality $\dim_\Q A(L)=[L:\Q]$ holds. 
\end{proposition}
\begin{proof}
We define the map $A(L)\to\prod_{\varrho\in R}L^{C_G(\varrho)}$ by sending $g\in A(L)$ 
to $\bigl(g(\varrho)\bigr)_{\varrho\in R}$. Indeed, for $\varrho\in R$ and $\sigma\in C_G(\varrho)$, 
the $G$-invariance of $g$ implies that $\sigma\bigl(g(\varrho)\bigr)=g(\sigma\varrho\sigma^{-1})=g(\varrho)$, 
hence $g(\varrho)$ is $C_G(\varrho)$-invariant. 
Conversely, for $(\xi_\varrho)_\varrho\in\prod_{\varrho\in R}L^{C_G(\varrho)}$, 
we define $g\colon G\to L$ by 
\[g(\sigma\varrho\sigma^{-1})\coloneqq \sigma(\xi_\varrho)\quad 
\text{for $\sigma\in G$ and $\varrho\in R$}. \]
This is well-defined since $\xi_\varrho$ is $C_G(\varrho)$-invariant, 
and the function $g$ is $G$-invariant. 
These two maps are inverse to each other. 

The latter part follows  from the equality $\dim_\Q L^{C_G(\varrho)}=\#G/\# C_G(\varrho)=\#[\varrho]$, 
where 
\begin{equation}\label{eq:conjugacy_rho}
[\varrho]:=\{\sigma \varrho \sigma^{-1}\mid \sigma\in G\}
\end{equation}
denotes the conjugacy class of $\varrho$. 
\end{proof}

\begin{remark}
Since $A(L)\cong \cP^\cA_L$, we have obtained an alternative proof of 
the equality $\dim_\Q\cP^\cA_L=[L:\Q]$. 
\end{remark}

In \cite{Rosen20}, it is implicitly noted that $\cP^\cA_0$ forms a $\Q$-subalgebra of $\cA$. 
% (as a consequence of the result stating that $\cP^\cA_0$ is obtained by the image of the ring of de Rham periods of Artin motives under the $\mathcal{A}$-valued period map).
We give a proof of it for the convenience of the reader.

\begin{proposition}\label{prop:Q-algebra cP^cA_0}
The subset $\cP^\cA_0$ of $\cA$ is a $\Q$-subalgebra, 
which is isomorphic to the direct limit $\varinjlim_{L}A(L)$, 
where $L$ runs over all finite Galois extensions of $\Q$ and 
the transition maps $A(L)\to A(L')$ are induced from the natural maps $\Fun(G,L)\to\Fun(G',L')$ 
for inclusions $L\subset L'$, {where $G'$ denotes the Galois group of $L'/\Q$. }
\end{proposition}
\begin{proof}
For any inclusion $L\subset L'$ of Galois extensions over $\Q$, 
we show that there is a commutative diagram 
\[\begin{CD}
\Fun(G,L) @>\ev>> \cA_L \\ 
@VVV @VVV \\
\Fun(G',L') @>\ev>> \cA_{L'}
\end{CD}\]
of $G'$-equivariant maps.
Then the statement follows by taking $G'$-invariant parts 
and by passage to the direct limit. 

The left vertical map is given by the composition with the natural surjection $G'\twoheadrightarrow G$ 
and the inclusion $L\hookrightarrow L'$, and the right vertical arrow is defined 
similarly to the map $\cA\to\cA_L$ of Proposition \ref{prop:cA_L^G}. 
The commutativity of the above diagram can be verified as follows. 
Take $g\colon G\to L$ and let $g'$ be the composition 
$G'\twoheadrightarrow G\xrightarrow{g}L\hookrightarrow L'$. 
For $\frp\in\bP_L$ such that ``$g(\phi_\frp)\bmod\frp$'' is defined, 
let $\frp'$ be an arbitrary prime of $L'$ lying above $\frp$. 
Then we have $\phi_{\frp'}|_L=\phi_\frp$ by the definition of Frobenius automorphisms, 
and hence $g'(\phi_{\frp'})=g(\phi_\frp)$, which shows the desired commutativity. 
\end{proof}

It is also easily checked that
\[\varinjlim_{L}A(L)\cong A(\overline{\Q})\coloneqq
\Fun_{\mathrm{cont}}(\Gal(\overline{\Q}/\Q),\overline{\Q})^{\Gal(\overline{\Q}/\Q)}, \]
where $\Fun_{\mathrm{cont}}(\Gal(\overline{\Q}/\Q),\overline{\Q})$ denotes 
the ring of continuous functions $\Gal(\overline{\Q}/\Q)\to\overline{\Q}$ 
with respect to the Krull topology on $\Gal(\overline{\Q}/\Q)$ and the discrete topology 
on $\overline{\Q}$, equipped with a natural action of $\Gal(\overline{\Q}/\Q)$.

\begin{remark}\label{rem:motivic}
From the theory of motivic periods \cite[\S5.1]{Brown17}, the $\Q$-algebra $\mathcal{P}_0^{\mathcal{A}}$ is obtained as the image of the $\Q$-algebra $\mathcal{P}^{\mathfrak{dr}}_{\mathcal{AM}(\Q)}$ of de Rham periods of the Tannakian category $\mathcal{AM}(\Q)$ of Artin motives under the $\mathcal{A}$-valued period map (see \cite{Rosen} for further details). 
This is an analogy of the fact that the $\Q$-algebra $\overline{\Q}$ of algebraic numbers is the image of the $\Q$-algebra $\mathcal{P}^{\mathfrak{m}}_{\mathcal{AM}(\Q)}$ of motivic periods of $\mathcal{AM}(\Q)$ under the $\C$-valued period map \cite{Brown17}.
Although we do not use this construction in this paper, it plays an important role in understanding that the $\Q$-algebra $\mathcal{P}_0^{\mathcal{A}}$ (resp.~$\mathcal{P}_L^{\mathcal{A}}$) is the true analogue of $\overline{\Q}$ (resp.~$L$) in $\mathcal{A}$ (see \cite[\S4]{Rosen20} for more details).
\end{remark}

\begin{remark}\label{rem:mzv}
The ring $\mathcal{A}$ appears in the recent works on finite multiple zeta values introduced by Kaneko and Zagier \cite{KanekoZagier,Kaneko19}; they conjectured a certain explicit correspondence between the finite multiple zeta value, an element in $\mathcal{A}$, and the symmetric multiple zeta value, an element in the $\Q$-algebra generated by multiple zeta values modulo $\pi^2$.
%$\C$, modulo $2\pi i$. 
Similarly to finite algebraic numbers, finite (resp. symmetric) multiple zeta values are obtained as the image of the ring $\mathcal{P}^{\mathfrak{dr}}_{\mathcal{MT}(\Z)}$ of de Rham periods (resp. the ring $ \mathcal{P}^{\mathfrak{m}}_{\mathcal{MT}(\Z)}$ of motivic periods) of the Tannakian category $\mathcal{MT}(\Z)$ of mixed Tate motives over $\Z$ under the $\mathcal{A}$-valued (resp. the $\C$-valued) period map; see \cite{Rosen}.
In contrast to the fact that there are non-canonical isomorphisms $\mathcal{P}^{\mathfrak{dr}}_{\mathcal{MT}(\Z)}\overset{\sim}{\rightarrow} \mathcal{P}^{\mathfrak{m}}_{\mathcal{MT}(\Z)}$ (see \cite[\S2.9]{Brown14}), the $\Q$-algebra $\mathcal{P}^{\mathfrak{dr}}_{\mathcal{AM}(\Q)}$ ($\cong\mathcal{P}_0^{\mathcal{A}}$) is not isomorphic to the $\Q$-algebra $\mathcal{P}^{\mathfrak{m}}_{\mathcal{AM}(\Q)}$ 
($\cong \overline{\Q}$).
So, for finite algebraic numbers, we can not expect the similar correspondence as in the conjecture of Kaneko and Zagier.
\end{remark}

%%%%%%%%%%%%%%%%%%%%%%%%%%%%%%%%%%%%%%%%
\section{Characterization of Idempotents by Recurrent Sequences}
%In this section, we study the idempotents of the ring $\cP^\cA_L$.
%Expressing these idempotents in terms of recurrent sequences, 
%we derive a decomposition law of primes in $L/\Q$.

In this section, we first prove Theorem \ref{thm:intro} and then discuss an explicit method to compute the linear recurrent sequence $(a_m(C))_m$ for any conjugacy class $C$ in the Galois group $\Gal(L/\Q)$ such that 
\begin{equation}\label{eq:conjugacy_class_idempotent} 
a_p(C)\equiv \begin{cases} 
1 \mod p & (\text{if $\phi_\frp\in C$ for $\frp\mid p$}),\\ 
0 \mod p & (\text{otherwise})
\end{cases}
\end{equation}
holds for all but finitely many primes $p$.
Then we will illustrate it by some examples.

%%%%%%%%%%%%%%%%%%%%%%%%%%%%%%%%%%%%%%%%
\subsection{Proof of Theorem \ref{thm:intro}}%Idempotents and Splitting of Primes}
Motivated by the right-hand side of \eqref{eq:conjugacy_class_idempotent}, for $\varrho \in G$ and a prime $p$ which is unramified in $L/\Q$, we define
\[ I_{\varrho}(p) \coloneqq \begin{cases} 
1 & (\text{if} \ \phi_{\frp} \in [\varrho] \ \text{for $\frp\mid p$}), \\ 
0 & (\text{otherwise}),
\end{cases}\]
where $[\varrho]$ denotes the conjugacy class of 
$\varrho$ in $G$ as in \eqref{eq:conjugacy_rho}.
We also let 
\[e^\cA_\varrho \coloneqq(I_{\varrho}(p))_p \in \cA.\]

%Let us construct idempotents of $A(L)$ and $\cP^\cA_L$.

%\begin{definition}\label{def:idempotent}
%Let $A(L)\cong\prod_{\rho\in R} L^{C_G(\rho)}$ be the isomorphism given in 
%Proposition \ref{prop:product decomposition}. Then we let $e_\rho$ be 
%the idempotent in $A(L)$ corresponding to the unit element of the component $L^{C_G(\rho)}$ 
%via the above isomorphism. Explicitly, $e_\rho\colon G\to L$ is defined by 
%\[ e_\rho(\tau) \coloneqq \begin{cases} 
%1 & (\text{if} \ \tau \in [\rho]), \\ 
%0 & (\text{otherwise}),
%\end{cases}\]
%%where $[\rho]=\{\sigma \rho \sigma^{-1}\mid \sigma\in G\}$ denote the conjugacy class of 
%$\rho$ in $G$. Since $e_\rho \in A(L)$, we get
%\[e^\cA_\rho\coloneqq \ev(e_\rho)\in \cP^\cA_L.\] 
%\end{definition}

%Now it is easy to show Theorem \ref{thm:intro}: 

\begin{proof}[{Proof of Theorem \ref{thm:intro}}]
Define $e_\varrho\colon G\to L$ by 
\begin{equation}\label{eq:def_e_rho} 
e_\varrho(\tau) \coloneqq \begin{cases} 
1 & (\text{if} \ \tau \in [\varrho]), \\ 
0 & (\text{otherwise}).
\end{cases}
\end{equation}
One finds that $e_\varrho \in A(L)$.
By definition, we have that ${\rm ev}(e_\varrho) = e_\varrho^{\cA}$, which implies $e_\varrho^{\cA}\in \cP_L^{\cA}$.
Therefore, the result follows from Theorem \ref{thm:r_f vs P_L^A} (2).
%, there exists a monic polynomial $f$ and 
%a sequence $\ba=(a_m)_m\in\Rec(f;\Q)$ satisfying $r_f(\ba)=e_\rho^\cA$. 
%By construction, this means that 
%\[a_p\equiv\begin{cases}
%1 \mod p & (\text{if $\phi_\frp\in[\rho]$ for $\frp\mid p$}), \\
%0 \mod p & (\text{otherwise})
%\end{cases}\]
%holds for all but finitely many primes $p$. 
\end{proof}

It should be noted that the above element $e_\varrho\in A(L)$ becomes the idempotent (i.e., $e_\varrho^2=e_\varrho$) corresponding to the unit element of the component $L^{C_G(\varrho)}$ via the isomorphism $A(L)\cong\prod_{\varrho\in R} L^{C_G(\varrho)}$ given in Proposition \ref{prop:product decomposition}.

\subsection{Methodology}\label{sec:method}
%Our basic idea to compute a sequence $\boldsymbol{a}\in  {\rm Rec}(f;\Q)$ such that $e^\cA_\rho = r_f(\boldsymbol{a})$ is as follows.

{We are interested not only in the existence of a sequence $\ba\in\Rec(f;\Q)$ such that $e_\varrho^\cA = r_f(\ba)$, but also in finding such a sequence explicitly.}
%Not only the existence, we are interested in finding explicitly 
%a sequence $\ba\in\Rec(f;\Q)$ such that $e_\varrho^\cA = r_f(\ba)$.
In the following, we explain our method to compute the sequence $\ba$.

\begin{proposition}\label{prop:e_rho_linear_equation}
Let $L$ be a finite Galois extension over $\Q$ and $f$ a monic polynomial in $\Q[x]$ which decomposes into linear factors over $L$.
Suppose that all roots $\xi_1,\ldots,\xi_d$ of $f$ are simple and that $e_\varrho^{\cA}\in r_f\big({\rm Rec}(f;\Q)\big)$.
Then, there exists a solution $(z_1,\ldots,z_d)\in L^d$ to the linear equations
\begin{align}\label{eq:lin_eq}
\sum_{j=1}^{d}z_{j} \tau(\xi_{j})=
\begin{cases}
1 & (\tau\in [\varrho]), \\ 
0 & (\tau\not\in[\varrho]).
\end{cases}
\end{align}
With this solution, letting $a_m= \sum_{j=1}^{d}z_{j} \xi_j^m \ (m\ge0)$, we obtain
\[ e_\varrho^{\cA} =r_f\big((a_m)_m\big).\]
\end{proposition}
\begin{proof}
Since all roots of $f$ are simple, the set $\{(\xi_j^m)_m\mid j=1,\ldots,d\}$ is a basis of $\Rec(f;L)$ (see \eqref{eq:basis of Rec(f;L)}).
Accordingly, for any $\boldsymbol{a}\in{\rm Rec}(f;\Q)$, we can write $\boldsymbol{a}=\biggl(\displaystyle\sum_{j=1}^d z_j \xi_j^m\biggr)_m$ for some $z_j\in L$. 
For such $\boldsymbol{a}$, one has
\[
\begin{array}{cccccccc}
\Rec(f;\Q)&\overset{\psi_f}{\longrightarrow} & 
(L\otimes L)^{G} & \overset{\varphi}{\longrightarrow} & 
A(L),\\
\boldsymbol{a} & \longmapsto & 
\displaystyle\sum_{j=1}^d z_j \otimes \xi_j &\longmapsto & 
\Biggl(\tau \mapsto\displaystyle \sum_{j=1}^d z_j \tau(\xi_j)\Biggr).
\end{array}
\]
Hence the assumption $e_\varrho^{\cA}\in r_f\big({\rm Rec}(f;\Q)\big)$ shows the existence of solutions to the linear system \eqref{eq:lin_eq} in $L^d$, 
and for such a solution, we get $e_\varrho^{\cA} =r_f(\boldsymbol{a})$ as desired.
\end{proof}

If all roots of $f$ are simple and form a basis of $L$ over $\Q$ 
(equivalently, if $f$ is the minimal polynomial of a normal element of $L$ over $\Q$), 
then by Theorem \ref{thm:r_f vs P_L^A} (2), 
the map $r_f\colon\Rec(f;\Q)\to\cP^\cA_L$ is an isomorphism.
In this case, we always have a unique solution to the equation \eqref{eq:lin_eq}, 
which gives rise to a solution to $e_\varrho^{\cA} =r_f(\boldsymbol{a})$.
In the computation, we do not need to find the above solution $(z_1,\ldots,z_d)\in L^d$ explicitly.
For example, {in the case} $\varrho=\id$, let $x_{m}$ ($m \ge 0$) be the determinant
\begin{align*}
x_{m}=
\begin{vmatrix}
\xi_{1}^{m} & \xi_{2}^{m} & \cdots & \xi_{d}^{m} \\
\tau_{2}(\xi_{1}) & \tau_{2}(\xi_{2}) & \cdots & \tau_{2}(\xi_{d}) \\
\vdots & \vdots & \ddots & \vdots \\
\tau_{d}(\xi_{1}) & \tau_{d}(\xi_{2}) & \cdots & \tau_{d}(\xi_{d})
\end{vmatrix},
\end{align*}
where we set $\{\tau_{1}, \ldots , \tau_{d}\}=G$ with $\tau_{1}=\mathrm{id}$.
Then, by Cramer's rule, we have $e^{\cA}_{\rm id} = r_f\big((a_m)_m\big)$ with $a_{m}=x_{m}/x_{1}$.

%%%%%%%%%%%%%%%%%%%%%%%%%%%%%%%%%%%%%%%%
\subsection{Examples}\label{sec:ex}
Let us work out some examples of the solution to the equation $e^{\cA}_\varrho = r_f(\boldsymbol{a})$.
For simplicity, we write 
\[e^\cA\coloneqq e_\id^\cA.\] 
Note that a sequence $\ba\in\Rec(f;\Q) $ such that $e^\cA=r_f(\ba)$ satisfies the congruences \eqref{eq:splitting_completely} for all but finitely many $p$, and hence, gives an explicit description of the set ${\rm Spl}(L)$ of primes splitting completely in $L/\Q$.
%\[ a_p\equiv\begin{cases} 
%1 \mod p & (\text{if $p$ splits completely in $L$}), \\
%0 \mod p & (\text{otherwise})
%\end{cases}\]
%holds for all but finitely many $p$. 
We also remark that the argument in the ring $\cA$ does not explicitly determine 
a finite set of exceptional primes. 
This problem is discussed in the next section.

\begin{example}[Quadratic fields]
For a square-free integer $D$, let $L=\Q(\sqrt{D})$ and write $G=\{{\rm id},\tau\}$.
In this case, since $p$ splits completely in $L$ if and only if $\bigl(\frac{D}{p}\bigr)=1$, the idempotents $e^{\mathcal{A}}_{\varrho}\in \mathcal{P}_L^{\mathcal{A}}$ are simply written as 
\[ e^{\mathcal{A}} = e^{\mathcal{A}}_{\rm id}=\frac{1}{2}\left(1+\left(\frac{D}{p}\right)\right)_p \quad \mbox{and}\quad e^{\mathcal{A}}_{\tau}=\frac{1}{2}\left(1-\left(\frac{D}{p}\right)\right)_p .\]
%where we set $\bigl(\frac{D}{\boldsymbol{p}}\bigr)=\bigl(\frac{D}{p}\bigr)_p\in \mathcal{A}$.
Now suppose that $\xi\in L$ is a normal element of $L$.
Its minimal polynomial $f=(x-\xi)(x-\tau(\xi))=x^2-c_1x+c_2\in \Q[x]$ satisfies $c_1\neq0$.
Note that $\xi=\frac{c_1\pm \sqrt{c_1^2-4c_2}}{2}$.
With this, it holds that
\[ x_0=\begin{vmatrix} 1&1\\ \tau(\xi)&\xi\end{vmatrix}= \sqrt{c_1^2-4c_2}, \quad x_1=  \begin{vmatrix} \xi&\tau(\xi) \\ \tau(\xi)&\xi\end{vmatrix} = c_1\sqrt{c_1^2-4c_2}.\]
Then the solution to $e^{\mathcal{A}} = r_f(\ba)$ is given by $\ba\in {\rm Rec}(f;\Q)$ with the initial values $a_0=\frac{1}{c_1}, \ a_1=1$.
%Namely, for sufficiently large prime $p$, we have $a_p\equiv 1\bmod p$ if and only if $p$ splits completely in $L$.
%In this way, we obtain a characterization of the splitting primes in $L/\Q$ by the sequence $(a_m)_m$ (though there are exceptions for finitely many primes).
\end{example}

\begin{example}[Cyclic cubic fields]
Let $L$ be a cyclic cubic extension over $\Q$ and
$\tau$ the element of $G$ of order $3$, i.e., $G=\{1, \tau, \tau^{2}\}$.
Set $\tau_{j}=\tau^{j-1}$ for $j=1, 2, 3$.
Suppose that $\xi \in L$ is a normal element of $L$.
Let $f(x)=x^{3}-c_{1}x^{2}+c_{2}x-c_{3}$ be the minimal polynomial of $\xi$ over $\mathbb{Q}$ ($c_1\neq0$).
Then we see that
\begin{align*}
x_{0}&=3c_{2}-c_{1}^{2}, \\
x_{1}&=c_{1}(3c_{2}-c_{1}^{2}), \\
x_{2}&=4c_{1}^{2}c_{2}-3c_{1}c_{3}-2c_{2}^{2}-c_{1}^{4}.
\end{align*}
Therefore, the corresponding sequence $(a_{m})_m\in \Rec(f;\Q)$ to $e^{\mathcal{A}} \in \mathcal{P}_L^{\mathcal{A}}$ is determined by the initial values
\begin{align*}
a_{0}=\frac{1}{c_{1}}, \quad
a_{1}=1, \quad
a_{2}=\frac{4c_{1}^{2}c_{2}-3c_{1}c_{3}-2c_{2}^{2}-c_{1}^{4}}{c_{1}(3c_{2}-c_{1}^{2})}.
\end{align*}
%and the recurrence relation
%\begin{align*}
%a_{m+3}-c_{1}a_{m+2}+c_{2}a_{m+1}-%c_{3}a_{m}=0 \qquad (m \ge 0).
%\end{align*}
\end{example}

\begin{remark}
For every cyclic cubic field $L/\Q$, there is a rational number $t\in \Q$ such that $L$ is the splitting field of Shanks' polynomial \cite{S74}
\[ f_t(x)=x^3-tx^2-(t+3)x-1.\]
When $t\neq0$, the above result shows that the sequence $(a_{m})_m\in \Rec(f_t;\Q)$ with the initial values
\[a_0=\frac{1}{t},\quad a_1= 1,\quad a_2 = \frac{2}{t}+1+t\]
satisfies $e^{\mathcal{A}}=r_{f_t}( (a_m)_m) $.
Note that the splitting field of $f_0$ coincides with $\Q(\zeta_9+\zeta_9^{-1})$ which is the splitting field of both $f_3$ and $f_{54}$; see \cite{Hoshi}.
\end{remark}

As an example of a decomposition law, we prove Proposition \ref{prop:S_3_intro}, the case $L=\mathbb{Q}(\sqrt[3]{2}, e^{2\pi \sqrt{-1}/3})$ whose Galois group is the symmetric group $\mathfrak{S}_{3}$.

\begin{proof}[Proof of Proposition \ref{prop:S_3_intro}]
Let $\omega=e^{\frac{2\pi \sqrt{-1}}{3}}$.
The element $\xi=\omega+2^{\frac{1}{3}}+2^{\frac{2}{3}}\omega $ is a normal element of $L$ and its minimal polynomial coincides with $f(x)=x^{6}+3x^{5}+12x^{4}+25x^{3}+60x^{2}+51x+127$.
By Theorem \ref{thm:r_f vs P_L^A}, the map $r_f:\Rec(f;\Q)\rightarrow \cP_L^{\cA}$ is an isomorphism.

Define the automorphisms $\sigma_1,\sigma_2$ of the Galois group $G$ for the generators $2^{\frac13}$ and $\omega$ of $L$ by 
\[ \sigma_1\bigl(2^{\frac13}\bigr)= 2^{\frac13}\omega, \ \sigma_1(\omega)=\omega \ \mbox{and} \ \sigma_2\bigl(2^{\frac13}\bigr)= 2^{\frac13}, \ \sigma_2(\omega)=\omega^2 .\]
These are representatives of conjugacy classes of $G$: $G=[{\rm id}]\cup [\sigma_1] \cup [\sigma_2]$ with $\# [\sigma_1]=2, \#[\sigma_2]=3$.
All elements of $G$ are $\tau_1={\rm id}, \ \tau_2=\sigma_1,\ \tau_3=\sigma_1^2, \ \tau_4=\sigma_2, \ \tau_5=\sigma_1\sigma_2, \ \tau_6=\sigma_1^2\sigma_2$. 
Let $\xi_j=\tau_j(\xi)\ (j=1,2,\ldots,6)$, a basis of $L$.
The group homomorphism $\rho: {\rm Gal}(L/\Q)\rightarrow {\rm GL}_2(\C)$ given by
\[\sigma_1\mapsto \begin{pmatrix} \omega & 0 \\ 0 & \omega^{-1}\end{pmatrix} ,\quad \sigma_2 \mapsto\begin{pmatrix} 0 & 1 \\ 1 & 0\end{pmatrix} \]
is an irreducible Artin representation, and hence, for any primes $p$ which is unramified in $L/\Q$ we get 
\[ b_p=\Tr(\rho(\phi_\frp)) = \begin{cases} 2 & ( \phi_\frp \in [{\rm id}]\ \mbox{for $\frp\mid p$}), \\ -1 & (\phi_\frp \in [\sigma_1]\ \mbox{for $\frp\mid p$}), \\ 0 & (\phi_\frp \in [\sigma_2]\ \mbox{for $\frp\mid p$}).\end{cases}\]
Theorem \ref{thm:intro} shows that $(b_p)_p=2e_{\rm id}^{\cA}-e_{\sigma_1}^{\cA}\in \cP_L^{\cA}$.
By Theorem \ref{thm:r_f vs P_L^A} (2), there is a sequence $\ba=(a_m)_m\in \Rec(f;\Q)$ such that $r_f(\ba)=(b_p)_p$.
Similarly to Proposition \ref{prop:e_rho_linear_equation}, we now consider the system of linear equations
\begin{equation}\label{eq:linear_eq_S3}
\begin{pmatrix} \tau_1(\xi_1) & \cdots &\tau_1(\xi_6) \\ \tau_2(\xi_1)&\cdots & \tau_2(\xi_6) \\ \vdots & \cdots& \vdots \\ \tau_6(\xi_1) & \cdots & \tau_6(\xi_6) \end{pmatrix} \begin{pmatrix}z_1\\ z_2\\ \vdots \\ z_6 \end{pmatrix} = \begin{pmatrix} 2\\ -1\\ -1\\ 0 \\0\\0\end{pmatrix}.
\end{equation}
Note that, as a permutation on the set $\{\xi_1,\ldots,\xi_6\}$, the automorphisms $\sigma_1,\sigma_2$ are viewed as 
\[ \sigma_1 = \begin{pmatrix}1&2&3&4&5&6 \\2&3&1&5&6&4\end{pmatrix},\quad \sigma_2 = \begin{pmatrix} 1&2&3&4&5&6\\ 4&6&5&1&3&2\end{pmatrix},\]
so the coefficient matrix of \eqref{eq:linear_eq_S3} is
\[\begin{pmatrix} 
\xi_1&\xi_2&\xi_3&\xi_4&\xi_5&\xi_6\\
\xi_2&\xi_3&\xi_1&\xi_5&\xi_6&\xi_4\\
\xi_3&\xi_1&\xi_2&\xi_6&\xi_4&\xi_5\\
\xi_4&\xi_6&\xi_5&\xi_1&\xi_3&\xi_2\\
\xi_5&\xi_4&\xi_6&\xi_2&\xi_1&\xi_3\\
\xi_6&\xi_5&\xi_4&\xi_3&\xi_2&\xi_1
\end{pmatrix}.\]
Let $z=\frac{1}{54} (2 \xi_1-\xi_2-\xi_3-2 \xi_4-2 \xi_5+4 \xi_6)$. 
Then one can check that 
\[ (z_1,z_2,\ldots,z_6)=(\tau_1(z),\tau_2(z),\ldots,\tau_6(z))\]
is the unique solution to \eqref{eq:linear_eq_S3}. % as follows:
%\begin{align*}
%z_1&=\frac{1}{54} (2 \xi_1-\xi_2-\xi_3-2 \xi_4-2 \xi_5+4 \xi_6),\\
%z_2&=\frac{1}{54} (-\xi_1+2 \xi_2-\xi_3+4 \xi_4-2 \xi_5-2 \xi_6),\\
%z_3&=\frac{1}{54} (-\xi_1-\xi_2+2 \xi_3-2 \xi_4+4 \xi_5-2 \xi_6),\\
%z_4&=\frac{1}{54} (-2 \xi_1+4 \xi_2-2 \xi_3+2 \xi_4-\xi_5-\xi_6),\\
%z_5&=\frac{1}{54} (-2 \xi_1-2 \xi_2+4 \xi_3-\xi_4+2 \xi_5-\xi_6),\\
%z_6&=\frac{1}{54} (4 \xi_1-2 \xi_2-2 \xi_3-\xi_4-\xi_5+2 \xi_6).
%\end{align*}
With this, we obtain
\[ (a_m)_m = \Biggl(\sum_{j=1}^6 \tau_j(z) \xi_j^m \Biggr)_m \in \Rec(f;\Q),\]
which satisfies $a_p\equiv b_p\bmod p$ for all but finitely many primes $p$.
In particular, we see that the initial values are given by $(a_0,a_1,\ldots,a_5)=(0,2,-4,-3,20,-40)$, from which the desired result follows.
\end{proof}

Several values of $a_p\bmod p$ for $(a_m)_m$ defined in Proposition \ref{prop:S_3_intro} are listed as follows.
\[
\begin{array}{c|ccccccccccccccccc}\hline
p & 23 & 29 & 31 & 37 & 41 & 43 & 47 & 53 & 59 & 61 & 67 & 71 \\ \hline
a_p\bmod p& 0 & 0 & 2 & -1 & 0 & 2 & 0 & 0 & 0 & -1 & -1 & 0\\ \hline
\end{array}\]
In fact, one can check that these values agree with $b_p$. 
In particular, since $b_{31}=b_{43}=2$, two primes $31$ and $43$ split completely. 
However, due to the existence of a finite set of exceptional primes, 
we cannot deduce this fact from our argument.
This inconvenience will be gotten rid of in the next section 
by considering the ``$S$-integral refinement'' of the theory of finite algebraic numbers. 

%%%%%%%%%%%%%%%%%%%%%%%%%%%%%%%%%%%%%%%%
\section{Determination of Exceptional Primes}
In \S\ref{sec:method} we have explained how to compute the solution $\ba\in\Rec(f;\Q)$ 
to the equation $e^{\cA}_\varrho=r_f(\ba)$.
However, as we have seen above, this method cannot provide information 
about a finite set $S$ of exceptional primes. 
The goal of this section is to control the set $S$ 
by refining the construction given in \S2 over the ring of $S$-integers.

%%%%%%%%%%%%%%%%%%%%%%%%%%%%%%%%%%%%%%%%
\subsection{The Ring $\cA_{L,S}$}
For a finite set $S$ of prime numbers, let 
\[\Z_S\coloneqq \Z\bigl[p^{-1}\bigm| p\in S\bigr]\]
be the ring of $S$-integers. We set 
\[\cA_S\coloneqq \prod_{p\in\bP_\Q\setminus S}\F_p
\cong\Z_S\otimes_\Z\prod_{p\in\bP_\Q}\F_p, \]
which is a $\Z_S$-algebra.
One can show that
\[\varinjlim_S\Z_S=\Q\quad \text{and}\quad \cA=\varinjlim_S\cA_S, \]
where $S$ runs through {all finite sets} of prime numbers and transition maps are 
the injections $\Z_S\to\Z_{S'}$ and the projections $\cA_S\to\cA_{S'}$ for $S\subset S'$, 
respectively. 
From the latter equality, we have the canonical surjection $\cA_{S}\rightarrow \cA$ sending each element to its equivalence class.
We also set 
\[O_{L,S}\coloneqq O_L\otimes_\Z\Z_S,\qquad  
\cA_{L,S}\coloneqq \prod_{\frp\in\bP_L\setminus S_L}O_L/\frp
\cong O_{L,S}\otimes_{O_L}\prod_{\frp\in\bP_L}O_L/\frp, \]
where $S_L$ denotes the set of primes of $L$ lying above elements of $S$. 
Then we have 
\[\varinjlim_S O_{L,S}=L\quad \text{and}\quad \cA_L=\varinjlim_S\cA_{L,S}. \]
Note that each $\cA_{L,S}$ admits a natural $G$-action, 
which induces the $G$-action on $\cA_L$ by passage to the inductive limit. 

%%%%%%%%%%%%%%%%%%%%%%%%%%%%%%%%%%%%%%%%
\subsection{Modules and Homomorphisms over $O_{L,S}$}
In this subsection, we construct the $G$-equivariant homomorphisms 
\[\Rec(f;O_{L,S})
\overset{\psi_{f,S}}{\longrightarrow}O_{L,S}\otimes O_{L,S}
\overset{\varphi_S}{\longrightarrow}\Fun(G,O_{L,S})
\overset{\ev_S}{\longrightarrow}\cA_{L,S}\]
of $O_{L,S}$-modules, under some assumptions on $S$ and $f$. 

First let us consider the map 
\[\varphi_S\colon O_{L,S}\otimes O_{L,S}\longrightarrow \Fun(G,O_{L,S});\ 
\xi\otimes\eta\longmapsto\bigl(\tau\mapsto\xi\tau(\eta)\bigr),\]
which is the $S$-integral version of the isomorphism $\varphi\colon L\otimes L\to\Fun(G,L)$. 
This is a homomorphism of $O_{L,S}$-algebras, and is $G$-equivariant. 

\begin{proposition}\label{prop:varphi S}
If $S$ contains all primes which ramify in $L/\Q$, then the map $\varphi_S$ is an isomorphism.
\end{proposition}
\begin{proof}
Recall that the discriminant of the number field $L$ is defined by 
\begin{equation}\label{eq:def_disc_L}
\disc(L)=\bigl(\det\bigl(\tau(\omega_j)\bigr)_{\tau,j}\bigr)^2, 
\end{equation}
where $\{\omega_1,\ldots,\omega_n\}$ is a $\Z$-basis of $O_L$. 
Here, $\bigl(\tau(\omega_j)\bigr)_{\tau,j}$ denotes the $n\times n$ matrix 
whose $(i,j)$-entry is $\tau_i(\omega_j)$, with some permutation $\tau_1,\ldots,\tau_n$ 
of elements of $G$; its square is independent of the permutation. 

By Dedekind's theorem on ramification and the assumption on $S$, 
the discriminant $\disc(L)$ is invertible in $\Z_S$, hence in $O_{L,S}$. This means that the functions 
$\tau\mapsto\tau(\omega_j)$ for $j=1,\ldots,n$ form an $O_{L,S}$-basis of $\Fun(G,O_{L,S})$. 
On the other hand, the elements $1\otimes\omega_j$ form an $O_{L,S}$-basis 
of $O_{L,S}\otimes O_{L,S}$. 
Therefore, $\varphi_S$ maps a basis to a basis, hence is an isomorphism. 
\end{proof}

In the following, we always assume that $S$ contains all primes ramifying in $L/\Q$. 
Then we can also define the homomorphism of $O_{L,S}$-algebras 
\[\ev_S\colon \Fun(G,O_{L,S})\longrightarrow \cA_{L,S}\]
by 
\[\ev_S(g)\coloneqq \bigl(g(\phi_\frp)\bmod\frp\bigr)_\frp. \]
Note that the $L$-algebra homomorphism $\ev\colon\Fun(G,L)\to\cA_L$ 
given in Definition \ref{defn:ev} is obtained from the above maps by passage to the inductive limit 
$\varinjlim_S$. 

\begin{proposition}\label{prop:ev S}
The map $\ev_S\colon \Fun(G,O_{L,S})\longrightarrow \cA_{L,S}$ is $G$-equivariant and injective. 
\end{proposition}
\begin{proof}
This is shown in the proof of Proposition \ref{prop:ev}. 
\end{proof}

Finally, let us construct the $S$-integral version of $\psi_f\colon\Rec(f;L)\to L\otimes L$. 
Let $f(x)=x^d+c_1x^{d-1}+\cdots+c_d$ be a monic polynomial with coefficients in $\Z_S$. 
For any $\Z_S$-algebra $R$, let $\Rec(f;R)$ denote the $R$-module of sequences $(a_m)_{m}$ 
in $R$ satisfying the linear recurrence relation $a_m+c_1a_{m-1}+\cdots+c_da_{m-d}=0$ for all $m\ge d$. 

In the following, we assume that the monic polynomial $f\in \Z_S[x]$ of degree $d$ has $d$ distinct roots $\xi_1,\ldots,\xi_d$, 
all belonging to $L$. 
Note that $f$ is not necessarily irreducible over $\Q$ at this stage. %$d$ is not necessarily the extension degree $[L:\Q]$ at this stage.

\begin{proposition}\label{prop:alpha_j^m S}
\begin{enumerate}
\item For $j=1,\ldots,d$, the sequence $(\xi_j^m)_m$ is an element of $\Rec(f;O_{L,S})$. 
\item The elements $(\xi_j^m)_m$ for $j=1,\ldots,d$ form an $O_{L,S}$-basis of $\Rec(f;O_{L,S})$ 
if and only if the discriminant $\disc(f)=\prod_{i<j}(\xi_i-\xi_j)^2$ of the polynomial $f$ 
is invertible in $\Z_S$. 
\end{enumerate}
\end{proposition}
\begin{proof}
\begin{enumerate}
\item Since $O_{L,S}$ is the integral closure of $\Z_S$ in $L$, we have $\xi_j\in O_{L,S}$. 
The recurrence relation of $(\xi_j^m)_m$ is obvious. 

\item Note that the map $(a_m)_m\mapsto (a_0,\ldots,a_{d-1})$ gives an isomorphism 
$\Rec(f;R)\to R^d$ of $R$-modules. Thus $(\xi_j^m)_m$ ($j=1,\ldots,d$) forms a basis 
if and only if the Vandermonde determinant 
\[\det\begin{pmatrix}
1 & \xi_1 & \ldots & \xi_1^{d-1}\\
1 & \xi_2 & \ldots & \xi_2^{d-1}\\
\vdots&\vdots&\ddots&\vdots\\
1 & \xi_d & \ldots & \xi_d^{d-1}
\end{pmatrix}\]
is invertible in $O_{L,S}$. Since its square is equal to $\disc(f)$, the statement follows. \qedhere
\end{enumerate}
\end{proof}

By Proposition \ref{prop:alpha_j^m S}, supposing that $\disc(f)\in\Z_S^\times$, we can define the $O_{L,S}$-linear map 
\[\psi_{f,S}\colon\Rec(f;O_{L,S})\longrightarrow O_{L,S}\otimes O_{L,S};\ 
(\xi_j^m)_m\longmapsto 1\otimes \xi_j. \]
This is a $G$-equivariant map. 

\begin{proposition}\label{prop:psi_f S isom}
The following conditions are equivalent. 
\begin{enumerate}
\item The map $\psi_{f,S}$ is an isomorphism. 
\item All roots $\xi_1,\ldots,\xi_d$ of $f$ form a $\Z_S$-basis of $O_{L,S}$. 
\item $d=[L:\Q]$ and $\det\bigl(\Tr(\xi_i\xi_j)\bigr)_{i,j}\in\Z_S^\times$.
\end{enumerate}
\end{proposition}
\begin{proof}
It is obvious that (1) and (2) are equivalent. 

Let us show the equivalence of (2) and (3). 
Since each of them includes the condition $d=[L:\Q]$, we may assume it. 
If we put $X=O_{L,S}$ and $X'=\sum_{j=1}^d \Z_S\cdot\xi_j$, we have $X'\subset X$ and 
\[X'=X\iff O_{L,S}\cdot\det\bigl(\Tr(\xi_i\xi_j)\bigr)_{i,j}
=O_{L,S}\cdot\disc(L)\ \bigl(=O_{L,S}\bigr),\]
cf.~\cite[Chapter III, Corollary to Proposition 5]{Serre}. 
Therefore, the condition (2) holds if and only if $d=[L:\Q]$ and 
$\det\bigl(\Tr(\xi_i\xi_j)\bigr)_{i,j}$ is invertible in $O_{L,S}$. 
Since $\det\bigl(\Tr(\xi_i\xi_j)\bigr)_{i,j}$ obviously belongs to $\Z_S$, 
the latter condition is equivalent to (3). 
\end{proof}

\begin{remark}\label{rem:S123}
Given a finite Galois extension $L/\Q$, we have introduced the following conditions on a finite set $S$ of primes and a monic polynomial $f\in \Q[x]$ of degree $d$ with $d$ distinct roots $\xi_1,\ldots,\xi_d$ in $L$:
\begin{enumerate}
\item[(S1)] $S$ contains all primes ramifying in $L/\Q$. 
%\item[(S2)] $f$ belongs to $\Z_S[x]$, has $d$ distinct roots $\xi_1,\ldots,\xi_d$ in $L$, 
\item[(S2)] $f$ belongs to $\Z_S[x]$
and $\disc(f)\in\Z_S^\times$. 
\item[(S3)] $d=[L:\Q]$ and $\det\bigl(\Tr(\xi_i\xi_j)\bigr)_{i,j}\in\Z_S^\times$. 
\end{enumerate}
The conditions (S1) and (S2) already appear in Theorem \ref{thm:exceptional prime}.
The role of each condition is as follows.
(S1) ensures that $\varphi_S$ is an isomorphism and also is needed to define the map $\ev_S$. 
(S2) allows us to define the map $\psi_{f,S}$, and then, the new condition (S3) is necessary and sufficient 
for $\psi_{f,S}$ to be an isomorphism. 

For any finite Galois extension $L$ over $\Q$, 
one can construct a pair $(S,f)$ satisfying all of these conditions as follows. 
First let $f$ be the minimal polynomial of a normal basis $\xi_1,\ldots,\xi_d$ of $L/\Q$. 
%Then let $S$ be the set of primes satisfying at least one of the following conditions: 
Then let $S$ be the set of primes which ramify in $L$, divide the denominator of a coefficient of $f$, divide the numerator of $\disc(f)$, or divide the numerator of $\det\bigl(\Tr(\xi_i\xi_j)\bigr)_{i,j}$.
%(1) is ramified in $L$, 
%(2) divides the denominator of a coefficient of $f$, 
%(3) divides the numerator of $\disc(f)$, or 
%(4) divides the numerator of $\det\bigl(\Tr_{L/\Q}(\xi_i\xi_j)\bigr)_{i,j}$. 
\end{remark}

%%%%%%%%%%%%%%%%%%%%%%%%%%%%%%%%%%%%%%%%
\subsection{$S$-integers in $\cP^\cA_L$}

As in the previous subsection, let $S$ be a finite set of primes containing all ones ramifying in $L/\Q$. 
By taking the $G$-invariant part of 
\[O_{L,S}\otimes O_{L,S}\overset{\varphi_S}{\longrightarrow}\Fun(G,O_{L,S})
\overset{\ev_S}{\longrightarrow}\cA_{L,S},\]
we obtain $\Z_S$-algebra homomorphisms 
\[(O_{L,S}\otimes O_{L,S})^G
\overset{\varphi_S}{\longrightarrow}A(L,S)\coloneqq\Fun(G,O_{L,S})^G
\overset{\ev_S}{\longrightarrow}(\cA_{L,S})^G\cong\cA_{S}.\]
The last isomorphism $(\cA_{L,S})^G\cong\cA_{S}$ is shown in the same way as 
Proposition \ref{prop:cA_L^G}. 

\begin{definition}
For a finite Galois extension $L$ over $\Q$ and a finite set $S$ of primes containing 
all ones ramifying in $L/\Q$, we define a $\Z_S$-subalgebra $\cP^\cA_{L,S}$ of $\cA_{L,S}$ by 
\[\cP^\cA_{L,S}\coloneqq \ev_S\bigl(A(L,S)\bigr). \]
\end{definition}

From the commutative diagram 
\[\xymatrix{
A(L,S) \ar[r]^-{\ev_S}_-{\cong} \ar@{^{(}->}[d] & \cP^\cA_{L,S} \ar@{^{(}->}[r] \ar@{->}[d] & 
\cA_S \ar@{->>}[d] \\
A(L) \ar[r]^-{\ev}_-{\cong} & \cP^\cA_L \ar@{^{(}->}[r] & \cA
}\]
we see that the image of $\cP^\cA_{L,S}$ under the projection $\cA_S\to\cA$ 
is contained in $\cP^\cA_L$, and the induced map $\cP^\cA_{L,S}\to\cP^\cA_L$ is an injection. 
Therefore, we may view $\cP^\cA_{L,S}$ as a subring of $\cP^\cA_L$. 
We call it \emph{the ring of $S$-integers} in $\cP^\cA_L$. 

As in the case of finite algebraic numbers, every $S$-integer in $\cP_L^{\cA}$ is obtained from an $S$-integral sequence in 
$\Rec(f;\Z_S)$ for some $f\in \Z_S[x]$.
To see this, we first notice that, if $f(x)\in \Q[x]$ is a monic polynomial having distinct roots in $L$ and satisfying the condition (S2) in Remark \ref{rem:S123}, 
we have a map 
\[\psi_{f,S}\colon \Rec(f;\Z_S)\longrightarrow (O_{L,S}\otimes O_{L,S})^G, \]
which is the $G$-invariant part of $\psi_{f,S}\colon \Rec(f;O_{L,S})\to O_{L,S}\otimes O_{L,S}$. 
Then, similarly to the case over $\Q$ (cf.~Theorem \ref{thm:r_f vs P_L^A}), one can observe that 
the composition $\ev_S\circ\varphi_S\circ\psi_{f,S}$ is equal to the map 
\[r_{f,S}\colon \Rec(f;\Z_S)\longrightarrow\cP^\cA_{L,S};\ (a_m)_m\longmapsto (a_p\bmod p)_p. \]
Further, if $S$ satisfies the condition (S3) in Remark \ref{rem:S123}, then the above map $r_{f,S}$ is an isomorphism.
As a result, we obtain the $S$-integral refinement of Theorem \ref{thm:r_f vs P_L^A}.

\begin{theorem}\label{thm:S-refinement}
Let $S$ be a finite subset of primes and $f\in \Q[x]$ a monic polynomial of degree $d$ with $d$ distinct roots in $L$.
\begin{itemize}
\item[(1)] If $S$ and $f$ satisfy the conditions (S1) and (S2) in Remark \ref{rem:S123}, then we have that $r_{f,S}=\ev_S\circ\varphi_S\circ\psi_{f,S}$.
In particular, $r_{f,S}\big(\Rec(f;\Z_S)\big) \subset \cP_{L,S}^{\cA}$.
\item[(2)] If $S$ and $f$ satisfy the conditions (S1) to (S3) in Remark \ref{rem:S123}, then the map $r_{f,S}:\Rec(f;\Z_S)\rightarrow \cP_{L,S}^{\cA}$ is an isomorphism.
\end{itemize}
\end{theorem}

%%%%%%%%%%%%%%%%%%%%%%%%%%%%%%%%%%%%%%%%
\subsection{Decomposition Laws of Primes with Recurrent Sequences}
In this subsection, we discuss the determination problems of exceptional primes in our decomposition law of primes in $L$.

Let us prove Theorem \ref{thm:exceptional prime}.
We notice that for each $\varrho\in G$, the function $e_\varrho$ defined in \eqref{eq:def_e_rho}
lies in $A(L,S)$ for any $S$. 
Denote its image under the map $\ev_S$ by $e^{\cA}_{\varrho,S}\in\cP^\cA_{L,S}$. 
This element $e^{\cA}_{\varrho,S}$ coincides with $e^\cA_\varrho=\ev(e_\varrho)\in\cP^\cA_L$ studied in \S3 
under the inclusion $\cP^\cA_{L,S}\subset\cP^\cA_L$ mentioned above. 

\begin{proof}[Proof of Theorem \ref{thm:exceptional prime}]
From the assumption on $S$ and $f$, we have $r_{f,S}\big(\Rec(f;\Z_S)\big) \subset \cP_{L,S}^{\cA}$.
Since $\ba=(a_m)_m \in \Rec(f;\Z_S)$, we obtain the identity $e^{\cA}_{\varrho,S}=r_{f,S}(\ba)$ in $\cP^\cA_{L,S}$ for any $\rho \in C$, from which the desired result follows.
\end{proof}

Below, we write $e^{\cA}=e^{\cA}_{\rm id,S}$ and consider the equation 
\begin{equation}\label{eq:e^A_S} 
e^{\cA}=r_{f,S}(\boldsymbol{a})
\end{equation}
with $\ba\in\Rec(f;\Z_S)$ to get a characterization of the splitting primes in $L/\Q$. 
In some cases, it is convenient to work with a polynomial $f$ 
which is \emph{not} the minimal polynomial of a normal element (namely, use Theorem \ref{thm:S-refinement} (1)). 
Though $\psi_f$ is not an isomorphism for such $f$, 
it is sometimes possible to give an explicit set $S$ satisfying (S1) and (S2) in Remark \ref{rem:S123} 
and an explicit solution to \eqref{eq:e^A_S}. 

As examples of solutions to \eqref{eq:e^A_S}, we deal with two families: the cyclotomic field $\Q(\zeta_N)$ and the field $\Q(\zeta_N,{\sqrt[N]{M}})$, where $\zeta_N$ is a primitive $N$-th root of unity.

\begin{example}[Cyclotomic fields $\Q(\zeta_N)$]
Let $N$ be a positive integer and $\Q(\zeta_N)$ be the cyclotomic field. 
Then 
\[S=\{p\in\bP_\Q\mid \text{$p$ divides $N$}\}\]
satisfies the condition (S1). Moreover, let $f(x)=x^N-1$.
We have
\begin{align*}
\disc(f)&=\prod_{0\le i<j\le N-1}(\zeta_N^i-\zeta_N^j)^2
=(-1)^{N(N-1)/2}\prod_{\substack{0\le i,j\le N-1\\ i\ne j}}(\zeta_N^i-\zeta_N^j)\\
&=(-1)^{N(N-1)/2} \zeta_N^{N(N-1)/2}\Biggl(\prod_{k=1}^{N-1}(1-\zeta_N^k)\Biggr)^N
=(-1)^{N(N-1)/2} N^N, 
\end{align*}
since 
\[\prod_{k=1}^{N-1}(1-\zeta_N^k)=\frac{x^N-1}{x-1}\bigg|_{x=1}=N. \]
Thus the condition (S2) also holds for $S$ and $f$. 

Now let $\ba=(a_m)_m$ be the sequence given by 
\[a_m\coloneqq\begin{cases}
1 & (m\equiv 1\mod N), \\
0 & (m\not\equiv 1\mod N), 
\end{cases}\]
which obviously belongs to $\Rec(f;\Z_S)$. 
As an element of $\Rec(f;O_{L,S})$, it is also written as 
\[a_m=\frac{1}{N}\sum_{j=0}^{N-1}\zeta_N^{(m-1)j}
=\frac{1}{N}\sum_{j=0}^{N-1}\zeta_N^{-j}(\zeta_N^j)^m. \]
Hence we have 
\[\psi_{f,S}(\ba)=\frac{1}{N}\sum_{j=0}^{N-1}\zeta_N^{-j}\otimes \zeta_N^j,\qquad 
\varphi_S\circ\psi_{f,S}(\ba)(\tau)=\frac{1}{N}\sum_{j=0}^{N-1}\zeta_N^{-j}\tau(\zeta_N^j)
\quad (\tau\in G). \]
Now recall that each $\tau\in G$ is characterized by $\tau(\zeta_N)=\zeta_N^a$ 
with a unique $a\in(\Z/N\Z)^\times$. 
Then 
\[\varphi_S\circ\psi_{f,S}(\ba)(\tau)=\frac{1}{N}\sum_{j=0}^{N-1}\zeta_N^{-j}\zeta_N^{aj}
=\frac{1}{N}\sum_{j=0}^{N-1}\zeta_N^{(a-1)j}
=\begin{cases}
1 & (a=1\text{ in }(\Z/N\Z)^\times),\\
0 & (\text{otherwise}).
\end{cases}\]
Since $a=1$ in $(\Z/N\Z)^\times$ means that $\tau=\mathrm{id}$ in $G$, 
we see that $\varphi_S\circ\psi_{f,S}(\ba)=e_\mathrm{id}$, 
and $r_{f,S}(\ba)=e^\cA$ in $\cP^\cA_{L,S}$. 
Therefore, we obtain the equivalence 
\[\text{$p$ splits completely in $\Q(\zeta_N)$}\iff a_p\equiv 1\pmod{p}
\iff p\equiv 1\pmod{N}\]
for all primes $p$ not dividing $N$, 
which is a well-known result in the theory of cyclotomic fields. 
\end{example}

% \begin{remark}
% For $N\in\N$, it is interesting to ask whether there exists $\ba \in \Rec(\Phi_N;\Q)$ such that $r_{\Phi_N}(\ba)=e^{\cA}\in \cP_{\Q(\zeta_N)}^{\cA}$, where $\Phi_N(x)$ is the cyclotomic polynomial, 
% i.e., the minimal polynomial of $\zeta_N$.
% For example (in degree 4), there seems no $\ba \in \Rec(\Phi_{12};\Q)$ such that $r_{\Phi_{12}}(\ba)=e^{\cA}$, while the sequence $\ba $ in $\Rec(\Phi_{5};\Z)$ (resp.~in $\Rec(\Phi_{10};\Z)$) with the initial values $(a_0,a_1,a_2,a_3)=(-1,1,0,0)$ (resp.~$(1,1,0,0)$) satisfies $r_{\Phi_{N},\{5\}}(\ba)=e^{\cA}$ for $N=5$ or $N=10$.
% \end{remark}

\begin{example}[The fields $\Q(\zeta_N,{\sqrt[N]{M}})$]
Let $N$ be a positive integer and $M$ an integer for which the field 
$L=\Q(\zeta_N,\sqrt[N]{M})$ has degree $N$ over $\Q(\zeta_N)$. 
Note that this includes {the field considered} in Proposition \ref{prop:S_3_intro}. 

We set 
\begin{align*}
S&=\{p\in\bP_\Q\mid \text{$p$ divides $N$, $M$ or $M^j-1$ for some $j=1,\ldots,N-1$}\},\\
f(x)&=\prod_{j=0}^{N-1}(x^N-M^j). 
\end{align*}
The condition (S1) holds since only the prime divisors of $NM$ ramify in $L$. 
Moreover, after some computations, we obtain that 
\[\disc(f)
=N^{N^2}\cdot M^{N(N-1)^2/2}\cdot\prod_{\substack{0\le j,j'\le N-1\\ j\ne j'}}(M^j-M^{j'})^N, \]
which implies the condition (S2). 

The roots of $f$ are written as $\zeta_N^i M^{j/N}$ with $i,j=0,\ldots,N-1$. 
The Galois group is given by $G=\{\tau_{a,b}\mid a\in(\Z/N\Z)^\times,\,b=0,\ldots,N-1\}$, 
where 
\[\tau_{a,b}(\zeta_N)=\zeta_N^a,\quad \tau_{a,b}(M^{1/N})=\zeta_N^bM^{1/N},\text{ so }
\tau_{a,b}(\zeta_N^i M^{j/N})=\zeta_N^{ai+bj} M^{j/N}. \]

Now we define the sequence $\ba=(a_m)_m$ by 
\begin{align*}
a_m&=\frac{1}{N^2}\sum_{i,j=0}^{N-1}\zeta_N^{(m-1)i}M^{(m-1)j/N}
=\frac{1}{N^2}\sum_{i,j=0}^{N-1}\zeta_N^{-i}M^{-j/N}(\zeta_N^i M^{j/N})^m\\
&=\begin{cases}
\frac{1}{N}\sum_{j=0}^{N-1}M^{j(m-1)/N} & (m\equiv 1\mod{N}), \\
0 & (m\not\equiv 1\mod{N}). 
\end{cases}
\end{align*}
This belongs to $\Rec(f;\Z_S)$, and satisfies 
\begin{align*}
\varphi_S\circ\psi_{f,S}(\ba)(\tau_{a,b})
&=\frac{1}{N^2}\sum_{i,j=0}^{N-1}\zeta_N^{-i}M^{-j/N}\tau_{a,b}(\zeta_N^i M^{j/N})
=\frac{1}{N^2}\sum_{i,j=0}^{N-1}\zeta_N^{-i}M^{-j/N}\cdot\zeta_N^{ai+bj} M^{j/N}\\
&=\frac{1}{N^2}\sum_{i,j=0}^{N-1}\zeta_N^{(a-1)i+bj}
=\begin{cases}
1 & ((a,b)=(1,0)),\\
0 & ((a,b)\ne (1,0)). 
\end{cases}
\end{align*}
In other words, the identity $e^\cA=r_{f,S}(\ba)$ holds in $\cP^\cA_{L,S}$. 
From this, we again reproduce the well-known result 
\[\text{$p$ splits completely in $L$}\iff \text{$p\equiv 1\bmod N$ and 
$M$ is an $N$-th power residue modulo $p$}\]
for $p\notin S$, since 
\[a_p=\frac{1}{N}\sum_{j=0}^{N-1}\bigl(M^{(p-1)/N}\bigr)^j
\equiv\begin{cases}
1 & (M^{(p-1)/N}\equiv 1\mod{p}), \\
0 & (\text{otherwise}) 
\end{cases}\]
for any prime $p\equiv 1\mod{N}$. 
\end{example}

%%%%%%%%%%%%%%%%%%%%%%%%%%%%%%%%%%%%%%%%%
\section*{Acknowledgments}
This work is partially supported by
JSPS KAKENHI Grant Number 18K03233, 18H01110, 18H05233, 20K14294 and 21K03185.
We are grateful to Prof.~Tomokazu Kashio and Ryutaro Sekigawa for answering our questions about cyclic cubic extension.

%%%%%%%%%%%%%%%%%%%%%%%%%%%%%%%%%%%%%%%%%%

\end{document}